\documentclass[12pt, a4]{amsart}
\usepackage{amsmath,amssymb}
\usepackage{fullpage}
\usepackage{stackrel}

\theoremstyle{plain}
\newtheorem{thm}{Theorem}[section]
\newtheorem{pro}[thm]{Proposition}
\newtheorem{lem}[thm]{Lemma}
\theoremstyle{definition}
\newtheorem{defn}[thm]{Definition}
\newtheorem{exa}[thm]{Example}
\newtheorem{rem}[thm]{Remark}

\numberwithin{equation}{section}

\newcommand{\bR}{{\mathbb R}}
\newcommand{\bN}{{\mathbb N}}
\renewcommand{\a}{\alpha}
\renewcommand{\b}{\beta}
\renewcommand{\c}{\gamma}
\renewcommand{\l}{\lambda}
\newcommand{\e}{\epsilon}

\newcommand{\Ra}{\Rightarrow}

\renewcommand{\d}{\delta}

\renewcommand{\(}{\left(}
\renewcommand{\)}{\right)}
\renewcommand{\[}{\left[}
\renewcommand{\]}{\right]}

\DeclareMathOperator{\Id}{Id}

\begin{document}
\title[Nonlocal Neumann BVPs]{Non-trivial solutions of local and non-local Neumann boundary-value problems}



\author{Gennaro Infante}
\address{Gennaro Infante, Dipartimento di Matematica e Informatica, Universit\`{a} della
Calabria, 87036 Arcavacata di Rende, Cosenza, Italy}%
\email{gennaro.infante@unical.it}%

\author{Paolamaria Pietramala}%
\address{Paolamaria Pietramala, Dipartimento di Matematica e Informatica, Universit\`{a} della
Calabria, 87036 Arcavacata di Rende, Cosenza, Italy}%
\email{pietramala@unical.it}%

\author[F. A. F. Tojo]{F. Adri\'an F. Tojo}
\address{F. Adri\'an F. Tojo, Departamento de An\'alise Ma\-te\-m\'a\-ti\-ca, Facultade de Matem\'aticas,
Universidade de Santiago de Com\-pos\-te\-la, 15782 Santiago de Compostela, Spain}%
\email{fernandoadrian.fernandez@usc.es}%

\thanks{Published in:  
Proceedings of the Royal Society of Edinburgh: Section A Mathematics, (2016) 146, pp~337-369. Doi:10.1017/S0308210515000499.}

\begin{abstract}
We prove new results on the existence, non-existence, localization and multiplicity of nontrivial solutions for perturbed Hammerstein integral equations. Our approach is topological and relies on the classical fixed point index. Some of the criteria involve a comparison with the spectral radius of some related linear operators. We apply our results to some boundary value problems with local and nonlocal boundary conditions of Neumann type. We illustrate in some examples the methodologies used.
\end{abstract}

\maketitle

\section{Introduction}
In this paper we discuss the existence, localization, multiplicity and non-existence of nontrivial solutions of the second order differential equation
\begin{equation}\label{eqintr}
 u''(t) +h(t,u(t))=0,\ t\in (0,1),
\end{equation}
subject to (local) Neumann boundary conditions  (BCs)
\begin{equation}\label{bcintrloc}
u'(0)=u'(1)=0,
\end{equation}
or to non-local BCs of Neumann type
\begin{equation}\label{bcintrnloc}
u'(0)=\alpha[u],\quad u'(1)=\beta[u],
\end{equation}
where $\alpha[\cdot]$, $\beta[\cdot]$ are
linear functionals given by Stieltjes integrals, namely
\begin{equation*}
\alpha[u]=\int_0^1 u(s)\,dA(s),\quad \beta[u]=\int_0^1 u(s)\,dB(s).
\end{equation*}
The local BVP \eqref{eqintr}-\eqref{bcintrloc} has been studied by Miciano and Shivaji in \cite{mish}, where the authors proved the existence of multiple positive solutions, by means of the quadrature technique; using Morse theory, Li~\cite{li} proved the existence of positive solutions and Li and co-authors \cite{lilima} continued the study of~\cite{li} and proved the existence of multiple solutions. Multiple positive solutions were also investigated by Boscaggin~\cite{Bos} via shooting-type arguments.

Note that, since $\lambda=0$ is an eigenvalue of the associated linear problem 
$$
u''(t)+\lambda u(t)=0,\quad u'(0)=u'(1)=0,
$$ 
the corresponding Green's function does not exist. Therefore we use a shift argument similar to the ones in \cite{xHanJMAA07,Tor, jwmz-na} and we study two related BVPs for which the Green's function can be constructed, namely
\begin{equation}\label{shiftint1}
- u''(t)- \omega^2 u(t)=f(t,u(t)):=h(t,u(t))- \omega^2 u(t),\quad u'(0)=u'(1)=0,
\end{equation}
and (with an abuse of notation)
\begin{equation}\label{shiftint2}
- u''(t)+ \omega^2 u(t)=f(t,u(t)):=h(t,u(t))+ \omega^2 u(t),\quad u'(0)=u'(1)=0.
\end{equation}
The BVPs \eqref{shiftint1} and  \eqref{shiftint2} have been recently object of interest by a number of authors, see for example \cite{BonPiz, Don, Feng,  Sun2, Sun,  Wang1, Wan,  Wang2, Yao2, Yao,  Zhan, zhli, Zhil}; in Section~\ref{Greenstudy} we study in details the properties of the associated Green's functions and we improve 
and complement some estimates that occur in earlier papers, see Remark~\ref{cintre}.

The formulation of the nonlocal BCs in terms of linear functionals is fairly general and includes, as special cases, multi-point and integral conditions, namely
$$
\a[u]=\sum_{j=1}^m\a_j u(\eta_j)\quad\text{or}\quad \a[u]=\int_{0}^1\phi(s)u(s)ds. 
$$
 We mention that multi-point and integral BCs are widely studied objects. The study of multi-point BCs was, as far as we know, initiated in 1908 by Picone~\cite{Picone}. Reviews on differential equations with BCs involving Stieltjes measures has been written in 1942 by Whyburn~\cite{Whyburn} and in 1967 by Conti~\cite{Conti}. We mention also the (more recent) reviews of Ma~\cite{rma},  Ntouyas~\cite{sotiris} and \v{S}tikonas~\cite{Stik} and the papers by Karakostas and Tsamatos~\cite{kttmna, ktejde} and by Webb and Infante~\cite{jwgi-lms}.
 
One motivation for studying nonlocal problems in the context of Neumann problems is that they occur naturally when modelling heat-flow problems. 

 For example the four point BVP
\begin{equation*}
 u''(t) +h(t,u(t))=0,\quad u'(0)=\a u(\xi),\; u'(1)=\beta u(\eta),\ {\xi,\eta}\in [0,1],
\end{equation*}
models a thermostat 
where two controllers at $t=0$ and $t=1$ add or remove heat
according to the temperatures detected by two sensors at $t=\xi$ and
$t=\eta$.
Thermostat  models of this type were studied in a number of papers, see for example
\cite{Cab1,Fan-Ma, gi-poit, gi-caa, gijwems, Kar-Pal, pp-gi-pp-aml-06, jwpomona, jwwcna04, jw-narwa} and references therein. In particular Webb~\cite{jw-narwa}
studied the existence of \emph{positive} solutions of the BVP 
\begin{equation*}
 u''(t) +h(t,u(t))=0,\quad u'(0)=\alpha[u],\; u'(1)=-\beta[u].
\end{equation*}
The methodology in~\cite{jw-narwa} is somewhat different from ours and relies on a careful rewriting of the associated Green's function, due to the presence of the term  $-\beta[u]$ in the BCs. 
The existence of solutions that \emph{change sign} have been investigated by Fan and Ma \cite{Fan-Ma}, in the case of the BVP
$$
 u''(t) +h(t,u(t))=0,\quad u'(0)=\a u(\xi),\; u'(1)=-\beta u(\eta),\ {\xi,\eta}\in [0,1],
$$
and in \cite{Cab1, gi-pp, gijwems} for the BVP
$$
 u''(t) +h(t,u(t))=0,\quad u'(0)=-\alpha[u],\; u'(1)=-\beta u(\eta),\ {\eta}\in [0,1].
$$
A common feature of the papers \cite{Cab1, Fan-Ma, gi-pp, gijwems} is that a direct construction of a Green's function is possible due to the term $-\beta u(\eta)$.

In Section \ref{secham} we develop a fairly general theory for the existence and multiplicity of nontrivial solutions of the perturbed Hammerstein integral equation of the form
\begin{equation}\label{perintro}
u(t)={\gamma}(t)\alpha[u]+{\delta(t)}{\beta}[u] +\int_0^1 k(t,s)g(s)f(s,u(s))\,ds,
\end{equation}
that covers, as \emph{special cases}, the BVP~\eqref{eqintr}, \eqref{bcintrnloc} and the BVP~\eqref{eqintr}- \eqref{bcintrloc} when $\alpha$ and $\beta$ are the trivial functionals.
We recall that the existence of \emph{positive solutions} of this type of integral equations has been investigated by Webb and Infante in~\cite{jwgi-lms}, under a non-negativity assumption on the terms $\gamma, \delta, k$, by working on a suitable cone of positive functions that takes into account the functionals $\alpha, \beta$.

In Section~\ref{secnon} we provide some sufficient conditions on the nonlinearity $f$ for the non-existence of solutions of the equation~\eqref{perintro}, this is achieved via an associated Hammerstein integral equation 
\begin{equation*}
u(t)= \int_0^1k_S(t,s)g(s)f(s,u(s))ds,
\end{equation*}
whose kernel $k_S$ is allowed to change sign and is constructed in the line of~\cite{jwgi-lms}, where the authors dealt with positive kernels.

In Section~\ref{seceigen} we provide a number of results that link the existence of nontrivial solutions of the equation~\eqref{perintro} with the spectral radius of some associated linear integral operators. The main tool here is the celebrated Krein-Rutman Theorem, combined with some ideas from the paper of Webb and Lan~\cite{jwkleig}; here due to the non-constant sign of the Green's function the situation is more delicate than the one in~\cite{jwkleig} and we introduce a number of different linear operators that yield different growth restrictions on the nonlinearity $f$.

In Section~\ref{secex} we illustrate the applicability of our theory in three examples,  two of which deal with solutions that change sign. The third example is taken from an interesting paper by Bonanno and Pizzimenti~\cite{BonPiz}, where the authors proved the existence, with respect to the parameter $\lambda$, of positive solutions of the following BVP 
$$
-u''(t)+u(t)=\l te^{u(t)},\quad  u'(0)=u'(1)=0.
$$
The methodology used in~\cite{BonPiz} relies on a critical point Theorem of Bonanno~\cite{Bon1}. Here we enlarge the range of the parameters and provide a sharper localization result. We also  prove a non-existence result for this BVP.

Our results complement the ones of \cite{jwgi-lms}, focusing the attention on the existence of solutions that are allowed to \emph{change sign}, in the spirit of the earlier works~\cite{gijwjmaa, gijwjiea, gijwems}. The approach that we use is topological, relies on classical fixed point index theory and we make use of ideas from the papers~\cite{Cab1,gijwjiea, jw-tmna, jwgi-lms, jwkleig}.

\section{Nonzero solutions of perturbed Hammerstein integral equations}\label{secham}
In this Section we  study the existence of solutions of the perturbed Hammerstein equations of the type
\begin{equation}\label{eqthamm}
u(t)={\gamma}(t)\alpha[u]+{\delta(t)}{\beta}[u] +\int_0^1 k(t,s)g(s)f(s,u(s))\,ds:=Tu(t),
\end{equation}
where
\begin{equation*}
 \alpha[u]=\int_0^1 u(s)\,dA(s),\;\;\beta[u]=\int_0^1 u(s)\,dB(s),
\end{equation*}
and $A$ and $B$ are functions of bounded variation.
If we set
\begin{equation*}
Fu(t):=\int_0^1 k(t,s)g(s)f(s,u(s))\,ds
\end{equation*}
we can write
\begin{equation*}
Tu(t)={\gamma}(t)\alpha[u]+{\delta(t)}{\beta}[u]+Fu(t),
\end{equation*}
that is, we consider $T$ as a perturbation of the simpler operator $F$.\par
 We work in  the space $C[0,1]$ of the continuous functions on $[0,1]$ endowed with the usual norm
 $\|w\|:=\max\{|w(t)|,\; t\; \in [0,1]\}$.\\
We make the following assumptions on the terms that occur in~\eqref{eqthamm}.
\begin{enumerate}
\item [$(C_{1})$] $k:[0,1] \times[0,1]\rightarrow \bR$ is measurable, and for every $\tau\in
[0,1] $ we have
\begin{equation*}
\lim_{t \to \tau} |k(t,s)-k(\tau,s)|=0 \;\text{    for almost every   } s \in [0,1] .
\end{equation*}{}
\item [$(C_{2})$] There exist a subinterval $[a,b] \subseteq [0,1]$, a function $\Phi \in L^{\infty}[0,1]$, and a constant $c_{1} \in (0,1]$ such that
\begin{align*}
|k(t,s)|\leq \Phi(s) \text{ for } &t \in [0,1] \text{ and almost every } \, s\in [0,1], \\ k(t,s) \geq c_{1}\Phi(s) \text{ for } &t\in
[a,b] \text{ and almost every } \, s \in [0,1] .
\end{align*}
\item [$(C_{3})$] $g\,\Phi \in L^1[0,1] $, $g(s) \geq 0$ for almost every $s\in [0,1]$, and $\int_a^b \Phi(s)g(s)\,ds >0$.{}
\item  [ $(C_{4})$] The nonlinearity $f:[0,1]\times (-\infty,\infty) \to [0,\infty)$ satisfies Carath\'{e}odory conditions, that is, $f(\cdot,u)$ is measurable for each fixed $u\in (-\infty,\infty)$ , $f(t,\cdot)$ is continuous for almost every $t\in [0,1]$, and for each $r>0$, there exists $\phi_{r} \in L^{\infty}[0,1] $ such that{}
\begin{equation*}
f(t,u)\le \phi_{r}(t) \;\text{ for all } \; u\in [-r,r],\;\text{ and almost every } \; t\in [0,1] .
\end{equation*}{}
\item[$(C_{5})$]{} $A, B$ are functions of bounded variation and $\mathcal{K}_{A}(s),\; \mathcal{K}_{B}(s)\geq 0$ for almost every $s\in [0,1]$, where
\begin{equation*}
\mathcal{K}_{A}(s):=\int_{0}^{1} k(t,s)\,dA(t)\text{ and } \mathcal{K}_{B}(s):=\int_{0}^{1} k(t,s)\,dB(t).
\end{equation*}{}
\item[$(C_{6})$] $\gamma \in C [0,1] , \; 0 \leq \alpha[\gamma] <1, \;\;\beta[\gamma]\geq 0.$\\
$ \text{There exists}\; c_{2} \in(0,1] \;\text{such that}\; \gamma(t) \geq c_{2}\|\gamma\| \;\text{for}\; t \in [a,b]$.
\item[$(C_{7})$] $\delta \in C[0,1] , \; 0 \leq \beta[\delta] <1,\;\;\alpha[\delta]\geq 0.$\\
$ \text{There exists}\; c_{3} \in(0,1] \;\text{such that}\; \delta(t) \geq c_{3}\|\delta\| \;\text{for}\; t \in [a,b]$.
\item [$(C_{8})$] $ D:=(1-\alpha[\gamma])(1-\beta[\delta]) -\alpha[\delta]\beta[\gamma]> 0$.{}
\end{enumerate}
From $(C_6)$-$(C_8)$ it follows that, for $\lambda\ge 1$,
$$
 D_{\lambda}:=(\lambda-\alpha[\gamma])(\lambda-\beta[\delta])-\alpha[\delta]\beta[\gamma]\geq D>0.
$$
We recall that a \emph{cone} $K$ in a Banach space $X$  is a closed
convex set such that $\lambda \, x\in K$ for $x \in K$ and
$\lambda\geq 0$ and $K\cap (-K)=\{0\}$.
The assumptions above allow us to work in the cone
$$
K:=\{u\in C[0,1]:\ \min_{t\in [a,b]}u(t)\ge c\|u\|,\ \alpha[u],\beta[u]\ge0\}
$$
where $c=\min\{c_1,c_2,c_3\}$.

Note that we have
$$K=K_0 \cap\{u\in C[0,1]: \alpha[u] \geq 0\}\cap\{u\in C[0,1]: \beta[u] \geq 0\},$$
where $$K_{0}:=\{u\in C[0,1]: \min_{t \in [a,b]}u(t)\geq c \|u\|\}.$$ 

The functions in $K_0$ are positive on the
subset $[a,b]$ but are allowed to change sign in $[0,1]$.
The cone $K_0$ is similar to
a cone of \emph{non-negative} functions
 first used by Krasnosel'ski\u\i{}, see e.g. \cite{krzab}, and D.~Guo, see e.g. \cite{guolak}. $K_0$ has been introduced by Infante and Webb in \cite{gijwjiea}
and later used in~\cite{ac-gi-at-bvp, Cab1, Fan-Ma, dfgior1, dfgior2,  giems, gi-pp1, gi-pp, gijwjmaa, gijwnodea, gijwems, nietopim}.
The cone $K$ allows the use of signed measures, taking into account two functionals. In the case of one functional this has been done in~\cite{Cab1}, where the authors dealt also with nontrivial solutions of the perturbed integral equation
\begin{equation}\label{one-pert-intro}
u(t)={\gamma}(t)\alpha[u] +\int_0^1 k(t,s)g(s)f(s,u(s))\,ds.
\end{equation}
In~\cite{Cab1} the authors work in the cone 
$K_0 \cap\{u\in C[0,1]: \alpha[u] \geq 0\}$, extending earlier the results in~\cite{gijwems} to the case of signed measures and the ones from~\cite{jwgi-nodea-08} to the context of nontrivial solutions. Clearly \eqref{one-pert-intro} is a special case of \eqref{eqthamm} and, by considering $\beta$ the trivial functional, we have
$K=K_0 \cap\{u\in C[0,1]: \alpha[u] \geq 0\}$.

A similar observation holds for the Hammerstein case 
\begin{equation}\label{ham-intro}
u(t)=\int_0^1 k(t,s)g(s)f(s,u(s))\,ds,
\end{equation}
studied in \cite{giems, gijwjmaa, gijwjiea} by means of the cone $K=K_0$. We mention that multiple solutions of \eqref{ham-intro} were investigated in the case of symmetric, sign changing kernels by Faraci and Moroz~\cite{farmor} by variational methods.

We also stress that, if we denote by $P$ the cone of positive functions, namely 
$$
P:=\{u\in C[0,1]:\ u(t)\geq 0, t\in [0,1]\},
$$
and consider $K\cap P$, we regain the cone of positive functions introduced by Webb and Infante in \cite{jwgi-lms}.

First of all we prove that $T$ leaves $K$ invariant and is compact.
\begin{lem}
The operator \eqref{eqthamm} maps $K$ into $K$ and is compact.
\end{lem}

\begin{proof}
Take $u\in K$ such that $\| u\| \leq r$. First of all, we observe that  $Tu(t)\geq 0$ for $t\in [a,b]$. We have, for 
$t\in [0,1]$,
\begin{equation*}
| Tu(t)|\leq |\gamma(t)|\alpha[u]+|\delta(t)|\beta[u]+\int_{0}^{1}|k(t,s)|g(s)f(s,u(s))\,ds,
\end{equation*}
therefore, taking the supremum on $t\in [0,1]$, we get
\begin{equation*}
\| Tu\| \leq \| \gamma\| \alpha[u]+\|\delta\|\beta [u]+\int_{0}^{1}\Phi
(s)g(s)f(s,u(s))\,ds,
\end{equation*}%
and, combining this fact with $(C_2)$, $(C_6)$ and $(C_7)$,
\begin{eqnarray*}
\min_{t\in [a,b]}Tu(t) &\geq &c_{2}\| \gamma \| \alpha[u]+c_{3}\| \delta\| \beta[u]+c_{1}\int_{0}^{1}\Phi(s)g(s)f(s,u(s))\,ds \\
&\geq &c \Vert Tu\Vert .
\end{eqnarray*}%
Furthermore, by $(C_3)$ and $(C_5)$-$(C_7)$,
$$
\alpha [Tu]=\alpha[\gamma]\alpha[u]+\alpha[\delta]{\beta}[u]+\int_0^1 \mathcal{K}_{A}(s)g(s)f(s,u(s))\,ds \geq 0
$$
and
$$
\beta [Tu]=\beta[\gamma]\alpha[u]+\beta[\delta]{\beta}[u]+\int_0^1 \mathcal{K}_{B}(s)g(s)f(s,u(s))\,ds \geq 0.
$$
Hence we have $Tu\in K$.\\ 
Moreover, the map $T$ is compact since it is sum of three compact maps: the compactness of $F$ is well-known and, since $\gamma
$ and $\delta$ are continuous, the perturbation $\gamma (t)\alpha[u]+\delta(t)\beta[u]$ maps bounded sets into
bounded subsets of a finite dimensional space.
\end{proof}

For $\rho>0$ we define the following open subsets of $K$:
\begin{equation*}
  K_{\rho}:=\{u \in K: \|u\|<{\rho}\}, \;\;
  V_{\rho}:=\{u \in K: \min_{t\in[a,b]}u(t)<{\rho}\}.
\end{equation*}
We have  $K_{\rho}\subset V_{\rho}\subset K_{\rho/c}$.\\
We recall some useful facts concerning real $2 \times 2$ matrices:

\begin{defn}\cite{jwgi-lms}
A $\;2 \times 2$ matrix $\mathcal{Q}$ is said to be order preserving (or non-negative) if $p_{1}\geq p_{0}$, $q_{1}\geq q_{0}$
imply
\begin{equation*}
\mathcal{Q}
\begin{pmatrix}
  p_{1} \\
  q_{1}
\end{pmatrix}%
\geq \mathcal{Q}
\begin{pmatrix}
  p_{0} \\
  q_{0}
\end{pmatrix},
\end{equation*}
in the sense of components.
\end{defn}
We have the following property, as stated in \cite{jwgi-lms}, whose proof is straightforward.
\begin{lem}\label{lematrix2}
Let
\begin{equation*}
\mathcal{Q}=
\begin{pmatrix}
  a & -b \\
  -c & d
\end{pmatrix}
\end{equation*}
with $a,b,c,d\geq 0$ and $\det \mathcal{Q}> 0$. Then $\mathcal{Q}^{-1}$ is order preserving.
\end{lem}

\begin{rem} \label{rem1}
It is a consequence of Lemma \ref{lematrix2} that if
 \begin{equation*}
\mathcal{N}=
\begin{pmatrix}
  1-a & -b\\
  -c & 1-d
\end{pmatrix},
\end{equation*}
 satisfies the hypotheses of Lemma \ref{lematrix2}, $p \geq 0, q \geq 0$ and $\mu>1$ then
$$
\mathcal{N}_\mu^{-1}\begin{pmatrix}
  p \\
  q
\end{pmatrix}\leq \mathcal{N}^{-1}\begin{pmatrix}
  p \\
  q
\end{pmatrix},
$$
where
 \begin{equation*}
\mathcal{N}_\mu=
\begin{pmatrix}
  \mu-a & -b\\
  -c & \mu-d
\end{pmatrix}.
\end{equation*}
\end{rem}
The next Lemma summarises some classical results regarding the fixed point index, for more details see~\cite{amann,guolak}.
If $\Omega$ is a open bounded subset of a cone $K$ (in the relative
topology) we denote by $\overline{\Omega}$ and $\partial \Omega$
the closure and the boundary relative to $K$. When $\Omega$ is an open
bounded subset of $X$ we write $\Omega_K=\Omega \cap K$, an open subset of
$K$.
\begin{lem} 
Let $\Omega$ be an open bounded set with $0\in \Omega_{K}$ and $\overline{\Omega}_{K}\ne K$. Assume that $F:\overline{\Omega}_{K}\to K$ is
a compact map such that $x\neq Fx$ for all $x\in \partial \Omega_{K}$. Then the fixed point index $i_{K}(F, \Omega_{K})$ has the following properties.
\begin{itemize}
\item[(1)] If there exists $e\in K\setminus \{0\}$ such that $x\neq Fx+\lambda e$ for all $x\in \partial \Omega_K$ and all $\lambda
>0$, then $i_{K}(F, \Omega_{K})=0$.
\item[(2)] If  $\mu x \neq Fx$ for all $x\in \partial \Omega_K$ and for every $\mu \geq 1$, then $i_{K}(F, \Omega_{K})=1$.
\item[(3)] If $i_K(F,\Omega_K)\ne0$, then $F$ has a fixed point in $\Omega_K$.
\item[(4)] Let $\Omega^{1}$ be open in $X$ with $\overline{\Omega^{1}}\subset \Omega_K$. If $i_{K}(F, \Omega_{K})=1$ and
$i_{K}(F, \Omega_{K}^{1})=0$, then $F$ has a fixed point in $\Omega_{K}\setminus \overline{\Omega_{K}^{1}}$. The same result holds if
$i_{K}(F, \Omega_{K})=0$ and $i_{K}(F, \Omega_{K}^{1})=1$.
\end{itemize}
\end{lem}

The following Proposition will be useful in the sequel, we give the proof for completeness.
\begin{pro}
Let $\omega\in L^1[0,1]$ and denote 
$$
\omega^+(s)=\max\{\omega(s),0\},\,\,\, \omega^-(s)=\max\{-\omega(s),0\}. 
$$
Then we have
$$
\left|\int_0^1 \omega(s)ds\right|\le\max\left\{\int_0^1 \omega^+(s)ds,\int_0^1 \omega^-(s)ds\right\}\le \int_0^1 |\omega(s)|ds.
$$
\end{pro}
\begin{proof}
 Observing that, since $\omega=\omega^+ - \omega^-$,
\begin{align*}\int_0^1 \omega(s)ds & =\int_0^1 \omega^+(s)ds-\int_0^1 \omega^-(s)ds\le\int_0^1 \omega^+(s)ds,\\
-\int_0^1 \omega(s)ds & =\int_0^1 \omega^-(s)ds-\int_0^1 \omega^+(s)ds\le\int_0^1 \omega^-(s)ds,\end{align*}
we get the first inequality, the second comes from the fact that $|\omega|=\omega^++\omega^-$.
\end{proof}
We now give a sufficient condition on the growth of the nonlinearity that provides that the index is $1$ on $K_{\rho}$.

\begin{lem}\label{thmind1}
Assume that 
\begin{enumerate}
  \item[$(\mathrm{I}_{\rho }^{1})$] \label{EqB} there exists $\rho>0$ such that
\begin{multline}\label{eqind1b}
   f^{ -\rho,{\rho}}\, \left(\sup_{t \in [0,1]  }\left\{ \left(\dfrac{|\gamma(t)|}{D}(1-\beta[\delta])+\dfrac{|\delta(t)|}{D}\beta[\gamma]\right)\int_0^1 \mathcal{K}_{A}(s)g(s)\,ds\right.\right.\\
+\left(\dfrac{|\gamma(t)|}{D}\alpha[\delta]+\dfrac{|\delta(t)|}{D}(1-\alpha[\gamma])\right) \int_0^1\mathcal{K}_{B}(s)g(s)\,ds\\
\left.\left.+\max\left\{\int_0^1k^+(t,s)g(s)\,ds,\int_0^1k^-(t,s)g(s)\,ds\right\} \right\} \right)<1.
\end{multline}{}
where
\begin{equation}\label{fsm}
f^{-\rho,{\rho}}:=\mathrm{ess } \sup \Bigl\{ \frac{f(t,u)}{ \rho}:\; (t,u)\in [0,1] \times[-\rho, \rho]\Bigr\}.
\end{equation}{}
\end{enumerate}
Then we have $i_{K}(T, K_{\rho})=1$.
\end{lem}

\begin{proof}
We show that $Tu \neq \lambda u$ for all $\lambda \geq 1$ when $u \in \partial K_{\rho}$, which implies that 
$i_{K}(T,K_{\rho})=1$. In fact, if this does not happen, then there exist $u$ with
$\|u\|=\rho$ and $\lambda \geq 1$ such that $ \lambda u(t)=Tu(t)$, that is 
\begin{equation}\label{equau}
\lambda u(t)={\gamma}(t)\alpha[u]+{\delta(t)}{\beta}[u]+Fu(t).
\end{equation}
Therefore we obtain
$$
\lambda\alpha[u]=\alpha[\gamma]\alpha[u]+\alpha[\delta]{\beta}[u]+\alpha[Fu]
$$
and
$$
\lambda\beta[u]=\beta[\gamma]\alpha[u]+\beta[\delta]{\beta}[u]+\beta[Fu].
$$ 
Thus we have
\begin{equation}\label{soldlambda2}
\begin{pmatrix}
\lambda-\alpha[\gamma] & -\alpha[\delta]\\
-\beta[\gamma] & \lambda-\beta[\delta]
\end{pmatrix}\begin{pmatrix}
\alpha[u]\\
\beta[u]
\end{pmatrix}
=
\begin{pmatrix}
\alpha[Fu]\\
\beta[Fu]
\end{pmatrix}.
\end{equation}
Note that the matrix that occurs in \eqref{soldlambda2} satisfies the hypothesis of Lemma \ref{lematrix2}, so its inverse is order preserving.
Then, applying its inverse matrix to both sides of \eqref{soldlambda2}, we have
\begin{equation*}
\begin{pmatrix}
\alpha[u]\\
\beta[u]
\end{pmatrix}
=\frac{1}{D_{\lambda}}
\begin{pmatrix}
\lambda-\beta[\delta] & \alpha[\delta]\\
\beta[\gamma] & \lambda-\alpha[\gamma]
\end{pmatrix}
\begin{pmatrix}
\alpha[Fu]\\
\beta[Fu]
\end{pmatrix}.
\end{equation*}
By Remark~\ref{rem1}, we obtain that 
\begin{equation}\label{sold1}
\begin{pmatrix}
\alpha[u]\\
\beta[u]
\end{pmatrix}
\leq \frac{1}{D}
\begin{pmatrix}
1-\beta[\delta] & \alpha[\delta]\\
\beta[\gamma] & 1-\alpha[\gamma]
\end{pmatrix}
\begin{pmatrix}
\alpha[Fu]\\
\beta[Fu]
\end{pmatrix}.
\end{equation}

Hence, from \eqref{equau} and \eqref{sold1} we get
\begin{align*}
\lambda |u(t)|\leq &
 \dfrac{|\gamma(t)|}{D}((1-\beta[\delta])\alpha[Fu]+\alpha[\delta]\beta[Fu])\\
& +\dfrac{|\delta(t)|}{D}((1-\alpha[\gamma])\beta[Fu])+\beta[\gamma]\alpha[Fu])+|Fu(t)|. 
\end{align*}
Taking the supremum over $[0,1] $ gives
\begin{multline*}
\lambda {\rho} \leq {\rho} f^{-\rho,\rho}\, \left(\sup_{t \in  [0,1] }\left\{\left(\dfrac{|\gamma(t)|}{D}(1-\beta[\delta])+\dfrac{|\delta(t)|}{D}\beta[\gamma]\right)\int_0^1 \mathcal{K}_{A}(s)g(s)\,ds\right.\right.\\
+\left(\dfrac{|\gamma(t)|}{D}\alpha[\delta]+\dfrac{|\delta(t)|}{D}(1-\alpha[\gamma])\right) \int_0^1\mathcal{K}_{B}(s)g(s)\,ds\\\left.\left.+\max\left\{\int_0^1k^+(t,s)g(s)\,ds,\int_0^1k^-(t,s)g(s)\,ds\right\} \right\}\right).
\end{multline*}
From \eqref{eqind1b} we obtain that $\lambda \rho <\rho$, contradicting the fact that $\lambda\geq 1$.
\end{proof}
\begin{rem}
In similar way as in \cite{jwgi-lms} (where the positive case was studied) we point out that a stronger (but easier to check) condition than $(\mathrm{I}_{\protect\rho }^{1})$ is given by the following.
 \begin{multline*}
 f^{-\rho,\rho}\, \left[ \left(\dfrac{\|\gamma\|}{D}(1-\beta[\delta])+\dfrac{\|\delta\|}{D}\beta[\gamma]\right)\int_0^1 \mathcal{K}_{A}(s)g(s)\,ds\right.\\
\left.+\left(\dfrac{\|\gamma\|}{D}\alpha[\delta]+\dfrac{\|\delta\|}{D}(1-\alpha[\gamma])\right) \int_0^1\mathcal{K}_{B}(s)g(s)\,ds +\frac{1}{m} \right]<1.
\end{multline*}
where
\begin{equation*}
\frac{1}{m}:=\sup_{t\in [0,1]}\left\{\max\left\{\int_0^1k^+(t,s)g(s)\,ds,\int_0^1k^-(t,s)g(s)\,ds\right\}\right\}.
\end{equation*}

Note that, since $\max\{k^+,k^-\}\leq |k|$, the constant $m$ provides a better estimate on the growth of the nonlinearity $f$ than the constant 
$$
\sup_{t\in [0,1]}\int_0^1|k(t,s)|g(s)\,ds,
$$
used in~\cite{ac-gi-at-bvp, Cab1, Fan-Ma, dfgior1, dfgior2, giems, gi-pp1, gi-pp, gijwjmaa, gijwnodea, gijwems, nietopim}.
\end{rem}
\begin{rem}
If the functions $\gamma, \delta, k$ are non-negative, we can work within the cone $K\cap P$, regaining the condition given in~\cite{jwgi-lms}, namely 
\begin{multline*}
 f^{0,\rho}\, \Bigl(\sup_{t \in [0,1]}
\Bigl\{ \bigl(
\dfrac{\gamma(t)}{D}(1-\beta[\delta])
+\dfrac{\delta(t)}{D}\beta[\gamma]\bigr)
\int_0^1 \mathcal{K}_{A}(s)g(s)\,ds\\
+\bigl(\dfrac{\gamma(t)}{D}\alpha[\delta]
+\dfrac{\delta(t)}{D}(1-\alpha[\gamma])\bigr) \int_0^1
\mathcal{K}_{B}(s)g(s)\,ds +\int_0^1 k(t,s)g(s)\,ds \Bigr\} \Bigr)<1,
\end{multline*}{}
where 
$$
f^{0,{\rho}}:=\mathrm{ess }\sup \Bigl\{ \frac{f(t,u)}{ \rho}:\; (t,u)\in [0,1] \times[0, \rho]\Bigr\}.
$$
\end{rem}

\begin{lem}\label{thmind0}
 Assume that \par
\begin{enumerate}
\item[$(\mathrm{I}_{\protect\rho }^{0})$] \label{EqC} There exists $\rho>0$ such that
\begin{multline}\label{eqind0}
f_{\rho, \rho/c} \, \left(\inf_{t \in [a,b]} \left\{\left(\dfrac{\gamma(t)}{D}(1-\beta[\delta])+\dfrac{\delta(t)}{D}\beta[\gamma]\right)\int_a^b \mathcal{K}_{A}(s)g(s)\,ds\right.\right.\\
\left.\left.+\left(\dfrac{\gamma(t)}{D}\alpha[\delta]+\dfrac{\delta(t)}{D}(1-\alpha[\gamma])\right)\int_a^b\mathcal{K}_{B}(s)g(s)\,ds +\int_a^b k(t,s)g(s)\,ds\right\}\right)>1,
\end{multline}
where
\begin{equation}\label{flr}
f_{\rho ,{\rho /c}}: =\mathrm{ess }\inf \left\{\frac{f(t,u)}{\rho }%
:\;(t,u)\in [a,b]\times [\rho ,\rho /c]\right\}.
\end{equation}{}
\end{enumerate}
Then we have $i_{K}(T, V_{\rho})=0.$
\end{lem}

\begin{proof}
Let $e(t)=\int_0^1 k(t,s)\,ds$ for $t\in [0,1]$. Then, according to $(C_{2})$, $(C_{3})$ and $(C_{5})$, we have $e\in K \setminus \{0\}$. We show that $u \neq Tu +{\lambda}e$ for all ${\lambda} \geq 0$
 and $u \in \partial V_{\rho}$ which implies that $i_{K}(T, V_{\rho})=0$.
In fact, if this does not happen, there are $u \in \partial V_{\rho}$ (and so  for $t \in [a,b]$ we have $\min u(t)=\rho$ and $\rho
\leq u(t) \leq \rho/c$) , and ${\lambda} \geq 0$ with
\begin{equation*}
u(t)=Tu(t)+{\lambda}e(t)=\gamma(t)\alpha[u]+{\delta(t)}{\beta}[u]+Fu(t)+{\lambda}e(t).
\end{equation*}
Applying $\alpha$ and $\beta$ to both sides of the previous equation we get
\begin{equation}\label{matrix2}
\begin{pmatrix}
1-\alpha[\gamma] & -\alpha[\delta]\\
-\beta[\gamma] & 1-\beta[\delta]
\end{pmatrix}
\begin{pmatrix}
\alpha[u]\\
\beta[u]
\end{pmatrix}
=
\begin{pmatrix}
\alpha[Fu]+\lambda \alpha[e]\\
\beta[Fu]+ \lambda \beta[e]
\end{pmatrix}\ge\begin{pmatrix}
\alpha[Fu]\\
\beta[Fu]
\end{pmatrix}.
\end{equation}
Note that the matrix that occurs in \eqref{matrix2} satisfies the hypothesis of Lemma \ref{lematrix2}, so its inverse is order preserving.
Then, applying the inverse matrix to both sides of \eqref{matrix2}, we have
\begin{equation*}
\begin{pmatrix}
\alpha[u]\\
\beta[u]
\end{pmatrix}
\geq \frac{1}{D}
\begin{pmatrix}
1-\beta[\delta] & \alpha[\delta]\\
\beta[\gamma] & 1-\alpha[\gamma]
\end{pmatrix}
\begin{pmatrix}
\alpha[Fu]\\
\beta[Fu]
\end{pmatrix}.
\end{equation*}

Therefore, for $t\in[a,b]$, we obtain
\begin{align*}
  u(t)\geq & \left( \dfrac{\gamma(t)}{D}(1-\beta[\delta])+\dfrac{\delta(t)}{D}\beta[\gamma]\right)\alpha[Fu]+\left(\dfrac{\gamma(t)}{D}\alpha[\delta]+\dfrac{\delta(t)}{D}(1-\alpha[\gamma])\right)\beta[Fu]\\
& +Fu(t)+{\lambda}e (t)\\ \geq &
\left( \dfrac{\gamma(t)}{D}(1-\beta[\delta])
+\dfrac{\delta(t)}{D}\beta[\gamma]\right) \int_a^b \mathcal{K}_{A}(s)g(s)f(s,u(s))\,ds\\ & +
\left(\dfrac{\gamma(t)}{D}\alpha[\delta]+\dfrac{\delta(t)}{D}(1-\alpha[\gamma])\right)\int_a^b \mathcal{K}_{B}(s)g(s)f(s,u(s))\,ds\\ & +\int_a^b k(t,s)g(s)f(s,u(s))\,ds.
\end{align*}
Taking the infimum for $t \in [a,b]$ then gives
\begin{multline*}
\rho=\min u(t)  \geq  \rho f_{\rho, \rho/c} \, \left(\inf_{t\in[a,b]}\left\{\left(\dfrac{\gamma(t)}{D}(1-\beta[\delta])+\dfrac{\delta(t)}{D}\beta[\gamma]\right)\int_a^b \mathcal{K}_{A}(s)g(s)\,ds\right.\right.\\  \left.\left.+\left(\dfrac{\gamma(t)}{D}\alpha[\delta]+\dfrac{\delta(t)}{D}(1-\alpha[\gamma])\right)\int_a^b \mathcal{K}_{B}(s)g(s)\,ds +\int_a^b k(t,s)g(s)\,ds\right\}\right),
\end{multline*}
 contradicting \eqref{eqind0}.
\end{proof}
 
\begin{rem}\label{easyidx0}
We point out, in similar way as in \cite{jwgi-lms}, that a stronger (but easier to check) condition than $(\mathrm{I}_{\protect\rho }^{0})$ is given by the following.
\begin{multline}\label{eqind0b}
f_{\rho, \rho/c} \, \Bigl(\Bigl(\dfrac{ c_{2}\|\gamma\|}{D}(1-\beta[\delta])+\dfrac{c_{3}\|\delta\|}{D}\beta[\gamma]\Bigr) \int_a^b \mathcal{K}_{A}(s)g(s)\,ds\\ +\Bigl(\Bigl(\dfrac{ c_{2}\|\gamma\|}{D}\alpha[\delta]+\dfrac{c_{3}\|\delta\|}{D}(1-\alpha[\gamma])\Bigr)\int_a^b\mathcal{K}_{B}(s)g(s)\,ds +\frac{1}{M(a,b)} \Bigr)>1,
\end{multline}{}
where
\begin{equation*} 
\frac{1}{M(a,b)} :=\inf_{t\in [a,b]}\int_{a}^{b}k(t,s)g(s)\,ds.
\end{equation*}
\end{rem}

We now combine the results above in order to prove a Theorem regarding the existence of one, two or three nontrivial solutions. The proof is a direct consequence of the properties of the fixed point index and is omitted. It is possible to state a result for the existence of four or more solutions, we refer to Lan~\cite{kljdeds} for similar statements.

\begin{thm}\label{thmcasesS} 
The integral equation~\eqref{eqthamm} has at least one non-zero solution in $K$ if one of the following conditions holds.

\begin{enumerate}
\item[$(S_{1})$] There exist $\rho _{1},\rho _{2}\in (0,\infty )$ with $\rho_{1}/c<\rho _{2}$ such that $(\mathrm{I}_{\rho _{1}}^{0})$ and $(\mathrm{I}_{\rho _{2}}^{1})$ hold.
\item[$(S_{2})$] There exist $\rho _{1},\rho _{2}\in (0,\infty )$ with $\rho_{1}<\rho _{2}$ such that $(\mathrm{I}_{\rho _{1}}^{1})$ and $(\mathrm{I}_{\rho _{2}}^{0})$ hold.
\end{enumerate}
The integral equation~\eqref{eqthamm} has at least two non-zero solutions in $K$ if one of the following conditions holds.

\begin{enumerate}
\item[$(S_{3})$] There exist $\rho _{1},\rho _{2},\rho _{3}\in (0,\infty )$ with $\rho _{1}/c<\rho _{2}<\rho _{3}$ such that $(\mathrm{I}_{\rho_{1}}^{0})$,  $(\mathrm{I}_{\rho _{2}}^{1})$ $\text{and}\;\;(\mathrm{I}_{\rho _{3}}^{0})$ hold.
\item[$(S_{4})$] There exist $\rho _{1},\rho _{2},\rho _{3}\in (0,\infty )$
with $\rho _{1}<\rho _{2}$ and $\rho _{2}/c<\rho _{3}$ such that $(\mathrm{I}_{\rho _{1}}^{1}),\;\;(\mathrm{I}_{\rho _{2}}^{0})$ $\text{and}\;\;(\mathrm{I}_{\rho _{3}}^{1})$ hold.
\end{enumerate}
The integral equation~\eqref{eqthamm} has at least three non-zero solutions in $K$ if one
of the following conditions holds.

\begin{enumerate}
\item[$(S_{5})$] There exist $\rho _{1},\rho _{2},\rho _{3},\rho _{4}\in (0,\infty )$ with $\rho _{1}/c<\rho _{2}<\rho _{3}$ and $\rho _{3}/c<\rho_{4}$ such that $(\mathrm{I}_{\rho _{1}}^{0}),$ $(\mathrm{I}_{\rho _{2}}^{1}),\;\;(\mathrm{I}_{\rho _{3}}^{0})\;\;\text{and}\;\;(\mathrm{I}_{\rho _{4}}^{1})$ hold.
\item[$(S_{6})$] There exist $\rho _{1},\rho _{2},\rho _{3},\rho _{4}\in (0,\infty )$ with $\rho _{1}<\rho _{2}$ and $\rho _{2}/c<\rho _{3}<\rho _{4}$ such that $(\mathrm{I}_{\rho _{1}}^{1}),\;\;(\mathrm{I}_{\rho_{2}}^{0}),\;\;(\mathrm{I}_{\rho _{3}}^{1})$ $\text{and}\;\;(\mathrm{I}_{\rho _{4}}^{0})$ hold.
\end{enumerate}
\end{thm}

\section{Some non-existence results}\label{secnon}
We now consider the auxiliary Hammerstein integral equation
\begin{equation}\label{ieS}
u(t)= \int_0^1k_S(t,s)g(s)f(s,u(s))ds:=Su(t),
\end{equation}
where the kernel $k_S$ is given by the formula
\begin{multline*}
k_{S}(t,s)= \dfrac{\gamma(t)}{D}\left[(1-\beta[\delta])\mathcal{K}_{A}(s)+\alpha[\delta]\mathcal{K}_{B}(s)\right]\\
 +\dfrac{\delta(t)}{D}\left[\beta[\gamma]\mathcal{K}_{A}(s)+(1-\alpha[\gamma])\mathcal{K}_{B}(s)\right] +k(t,s).
\end{multline*}
The operator $S$ shares a number of useful properties with $T$, firstly the cone invariance and compactness, the proof follows directly from $(C_1)$-$(C_8)$ and is omitted. 
\begin{lem}
The operator \eqref{ieS} maps $K$ into $K$ and is compact.
\end{lem}

A key property that is also useful is the one given by the following Theorem; the proof is similar to 
the one in~\cite[Lemma 2.8 and Therem 2.9]{jwgi-lms} and is omitted.
\begin{lem}\label{sandt}
The operators $S$ and $T$ have the same fixed points in $K$. Furthermore
if $u\ne Tu$ for $u\in\partial D_K$, then $i_K(T,D_K)=i_K(S,D_K)$.
\end{lem}

We define the constants
\begin{equation*}
\frac{1}{m_S}:=\sup_{t\in [0,1]}\left\{\max\left\{\int_0^1k_S^+(t,s)g(s)\,ds,\int_0^1k_S^-(t,s)g(s)\,ds\right\}\right\},
\end{equation*}
\begin{equation*} 
\frac{1}{M_S(a,b)}=\frac{1}{M_S} :=\inf_{t\in [a,b]}\int_{a}^{b}k_S(t,s)g(s)\,ds,
\end{equation*}
and we prove the following non-existence results.
\begin{thm}\label{noext}
 Assume that one of the following conditions holds: \par
$(1)$  $f(t,u)<m_S|u|$ for every $t\in [0,1]$ and $u\in\bR\backslash\{0\}$,\par
$(2)$  $f(t,u)>M_Su$ for every $t\in [a,b]$ and $u\in\bR^+$.\\
 Then the equations \eqref{eqthamm} and \eqref{ieS} have no non-trivial solution in $K$. 
\end{thm}

\begin{proof}
In view of Lemma~\ref{sandt} we prove the Theorem using the operator $S$.\par
$(1)$ Assume, on the contrary, that there exists $u\in K$, $u\not\equiv 0$ such that $u=Su$ and let $t_0\in [0,1]$ such that $\|u\|=|u(t_0)|$. Then we have
\begin{align*}
\|u\|= & |u(t_0)|=\left|\int_0^1k_S(t_0,s)g(s)f(s,u(s))ds\right|\\ \le & \max\left\{\int_0^1k_S^+(t_0,s)g(s)f(s,u(s))\,ds,\int_0^1k_S^-(t_0,s)g(s)f(s,u(s))\,ds\right\}\\
< & \max\left\{\int_0^1k_S^+(t_0,s)g(s)m_S|u(s)|\,ds,\int_0^1k_S^-(t_0,s)g(s)m_S|u(s)|\,ds\right\}\\
\le & \max\left\{\int_0^1k_S^+(t_0,s)g(s)\,ds,\int_0^1k_S^-(t_0,s)g(s)\,ds\right\}m_S\|u\|\le\|u\|,
\end{align*}
a contradiction.\par
$(2)$ Assume, on the contrary, that there exists $u\in K$, $u\not\equiv 0$ such that $u=Su$ and let $\eta\in [a,b]$ be such that $u(\eta)=\min_{t\in[a,b]}u(t)$. For $t\in [a,b]$ we have
\begin{align*}
u(t)= & \int_0^1 k_S(t,s)g(s)f(s,u(s))ds\geq \int_{a}^{b} k_S(t,s)g(s)f(s,u(s))ds\\ 
> & M_S \int_{a}^{b} k_S(t,s)g(s)u(s)ds.
\end{align*}
Taking the infimum  for $t\in [a,b]$, we have
$$
\min_{t\in[a,b]}u(t)> M_S \inf_{t\in[a,b]}\int_{a}^{b} k_S(t,s)g(s)u(s)\,ds.
$$
Thus we obtain
$$
u(\eta)>M_S u(\eta) \inf_{t\in[a,b]}\int_{a}^{b} k_S(t,s)g(s)\,ds =u(\eta),
$$
a contradiction.
\end{proof}

\section{Eigenvalue criteria for the existence of nontrivial solutions}\label{seceigen}
In this Section we assume the additional hypothesis that the functionals $\alpha$ and $\beta$ are given by \emph{positive measures}. 

In order to state our eigenvalue comparison results, we consider the following operators on $C[0,1]$.
$$
L\,u(t):=  \int_0^1| k_{S}(t,s)|g(s)u(s)\,ds,\quad \tilde L\,u(t):=  \int_a^b k^+_{S}(t,s)g(s)u(s)\,ds.
$$
By similar proofs of \cite[Lemma 2.6 and Theorem 2.7]{jwgi-lms}, we study the properties of those operators.

\begin{thm}\label{lcomp} 
The operators $L$ and $\tilde L$ are compact and map $P$ into $P\cap K$.
\end{thm}

\begin{proof}
 Note that the operators $L$ and $\tilde L$ map $P$ into $P$ (because they have a positive integral kernel) and are compact. We now show that they map $P$ into $P\cap K$. We do this for the operator $L$, a similar proof works for $\tilde L$.
 
Firstly we observe that 
\begin{align*}
|k_S(t,s)|\le & \dfrac{|\gamma(t)|}{D}((1-\beta[\delta])\mathcal{K}_{A}(s)+\alpha[\delta]\mathcal{K}_{B}(s))\\
 & +\dfrac{|\delta(t)|}{D}(\beta[\gamma]\mathcal{K}_{A}(s)+(1-\alpha[\gamma])\mathcal{K}_{B}(s)) +|k(t,s)|\\
 \le & \dfrac{\|\gamma\|}{D}((1-\beta[\delta])\mathcal{K}_{A}(s)+\alpha[\delta]\mathcal{K}_{B}(s))
 +\dfrac{\|\delta\|}{D}(\beta[\gamma]\mathcal{K}_{A}(s)+(1-\alpha[\gamma])\mathcal{K}_{B}(s))\\
 & +|k(t,s)|\\
 \le & \Upsilon(s)+\Phi(s)=:\Psi(s),
 \end{align*}
 where
$$\Upsilon(s)=\dfrac{\|\gamma\|}{D}\left((1-\beta[\delta])\mathcal{K}_{A}(s)+\alpha[\delta]\mathcal{K}_{B}(s)\right)
  +\dfrac{\|\delta\|}{D}\left(\beta[\gamma]\mathcal{K}_{A}(s)+(1-\alpha[\gamma])\mathcal{K}_{B}(s)\right).
$$
Moreover, we have, for $t\in [a,b],$
 \begin{align*}
 |k_S(t,s)|=k_S(t,s)\ge & \dfrac{c_2\|\gamma\|}{D}\left[(1-\beta[\delta])\mathcal{K}_{A}(s)+\alpha[\delta]\mathcal{K}_{B}(s)\right]\\ & +\dfrac{c_3\|\delta\|}{D}\left[\beta[\gamma]\mathcal{K}_{A}(s)+(1-\alpha[\gamma])\mathcal{K}_{B}(s)\right]+c_1\Phi(t) \ge c\Psi(s),
 \end{align*}
 and thus
\begin{equation}\label{cofks}
\min_{t\in [ a,b]}k_S(t,s) \geq c\Psi(s).
\end{equation}
Also we have $g\,\Psi \in L^1[0,1]$ and we obtain that, for $u\in P$ and $t\in [0,1]$,
\begin{equation*}
Lu(t)\leq\int_{0}^{1}\Psi(s)g(s)u(s)\,ds,
\end{equation*}
in such a way that, taking the supremum on $t\in [0,1]$, we get
\begin{equation*}
\Vert Lu\Vert \leq \int_{0}^{1}\Psi(s)g(s)u(s)\,ds.
\end{equation*}%
On the other hand, 
$$
\min_{t\in [a,b]}Lu(t) \geq c\int_{0}^{1}\Psi(s)g(s)u(s)\,ds 
\geq c \Vert Lu\Vert .
$$
Furthermore, since $\alpha$ and $\beta$ are given by positive measures,
$$
\alpha [Lu]=\int_0^1\int_{0}^{1}| k_S(t,s)|g(s)u(s)\,ds\,dA(t) \geq 0
$$
and
$$ \beta[Lu]=\int_0^1 \int_{0}^{1}| k_S(t,s)|g(s)u(s)\,ds\,dB(t) \geq 0.
$$
Hence we have $Lu\in K$. 
\end{proof}

We recall that $\lambda$ is an \textit{eigenvalue} of a linear operator $\Gamma$ with corresponding eigenfunction $\varphi$ if $\varphi \neq 0$ and ${\lambda}\varphi=\Gamma \varphi$. The reciprocals of nonzero
eigenvalues are called \emph{characteristic values} of $\Gamma$. 
We will denote the \textit{spectral radius} of $\Gamma$ by $r(\Gamma):=\lim_{n\to\infty}\|\Gamma^n\|^{\frac{1}{n}}$ and its \textit{principal characteristic value} (the reciprocal of the spectral radius) by $\mu(\Gamma):=1/r(\Gamma)$.\par

The following Theorem is analogous to the ones in \cite{jwgi-lms, jwkleig} and is proven by using the facts that the considered operators leave $P$ invariant, that $P$ is reproducing, combined with the well-known Krein-Rutman Theorem. The condition $(C_3)$ is used to show that $r(L)>0$.
\begin{thm}\label{specrad} The spectral radius of $L$ is non-zero and is an eigenvalue of $L$ with an eigenfunction in $P$. A similar result holds for $ \tilde L$.
\end{thm}

\begin{rem}
As a consequence of the two previous theorems, we have the above mentioned eigenfunction is in $P\cap K$.
\end{rem}

We utilize the following operator on $C[a,b]$ defined by, for $t\in[a,b]$,
$$
\bar{L} u(t):=  \int_a^b k^+_{S}(t,s)g(s)u(s)\,ds
$$
and the cone $P_{[a,b]}$ of positive functions in $C[a,b]$. 
 
In the recent papers \cite{jw-lms, jw-tmna}, Webb developed an elegant theory valid for $u_0$-positive linear operators relative to two cones. It turns out that our operator  $\bar{L}$ fits within this setting and, in particular, satisfies the assumptions of Theorem $3.4$ of \cite{jw-tmna}.  We state here a special case of Theorem $3.4$ of \cite{jw-tmna} that can be used for $\bar{L}$.
 
\begin{thm}\label{thmjeff}
Suppose that there exist $u\in P_{[a,b]}\setminus \{0\}$ and $\lambda>0$ such that $$\lambda u(t)\geq \bar{L}u(t),\ \text{for}\ t\in [a,b].$$ Then we have $r(\bar{L})\leq \lambda$.
\end{thm}
We define the following extended real numbers. 
\begin{equation}\label{fas}
\begin{split}
f^{0}=\varlimsup_{u\to 0}\frac{\mathrm{ess} \sup\limits_{t \in [0,1]} f(t,u)}{|u|},& \;\;
f_{0}=\varliminf_{u\to 0^+}\frac{\mathrm{ess} \inf\limits_{t \in [a,b]}f(t,u)}{u}, \\
f^{\infty}=\varlimsup_{|u| \to+\infty}\frac{\mathrm{ess} \sup\limits_{t \in [0,1]} f(t,u)}{|u|},&
\;\; f_{\infty}=\varliminf_{u \to+\infty}\frac{\mathrm{ess} \inf\limits_{t \in [a,b]}f(t,u)}{u}.
\end{split}
\end{equation}

In order to prove the following Theorem, we adapt some of the proofs of \cite[Theorems 3.2-3.5]{jwkleig} to this new context.
\begin{thm}\label{thmindeig}
We have the following.
\begin{enumerate}
\item[$(1)$] If $\;0\le f^{0}<\mu(L)$, then there exists $\rho_{0}>0$ such
that
$
i_{K}(T,K_{\rho})=1$ for each $\rho\in (0,\rho_{0}].$

\item[$(2)$] If $\;0\le f^{\infty}<\mu(L)$, then there exists $R_{0}>0$ such
that
$
i_{K}(T,K_{R})=1$ for each $R > R_{0}.
$
\item[$(3)$] If $\mu(\tilde L)<f_{0}\leq \infty$, then there exists
$\rho_{0}>0$ such that 
$
i_{K}(T,K_{\rho})=0
$
 for each 
$\rho\in (0,\rho_{0}].$
\item[$(4)$] If $\mu(\tilde L)<f_{\infty} \leq \infty$, then there
exists $R_{1}>0$ such that 
$
i_{K}(T,K_{R})=0$
for each $R \geq R_{1}.$
\end{enumerate}
\end{thm}

\begin{proof}
We show the statements for the operator $S$ instead of $T$, in view of Lemma~\ref{sandt}.

$(1)$ Let $\tau$ be such that $f^0\le \mu(L)-\tau$. Then there exists $\rho_0\in(0,1)$ such that for all $u\in[-\rho_0,\rho_0]$ and almost every $t\in [0,1]$ 
we have
\begin{equation*}
f(t,u)\le(\mu(L)-\tau)|u|.
\end{equation*}
Let $\rho\in (0,\rho_0]$.  We prove that $Su\ne\lambda u$ for $u\in\partial K_\rho$ and $\lambda\ge 1$, which implies $ i_{K}(S,K_{\rho})=1$. In fact, if we assume otherwise, then there exists $u\in\partial K_\rho$ and $\lambda\ge1$ such that $\l u=Su$. Therefore,
\begin{align*}
 |u(t)|\leq &\lambda |u(t)|=  |Su(t)|  =  \left|\int_0^1k_S(t,s)g(s)f(s,u(s))ds\right|\\
 \le & \int_0^1|k_S(t,s)|g(s)f(s,u(s))ds \le  (\mu(L)-\tau)\int_0^1|k_S(t,s)|g(s)|u(s)|ds\\ 
 =&(\mu(L)-\tau)L |u|(t).
 \end{align*}
Thus, we have,  for $ t\in [0,1]$,
\begin{align*}
|u(t)|\le & (\mu(L)-\tau)L[(\mu(L)-\tau)L|u|(t)]\\ &=(\mu(L)-\tau)^2L^2|u|(t)\le\cdots\le(\mu(L)-\tau)^nL^n|u|(t),
 \end{align*}
thus, taking the norms, $1\le(\mu(L)-\tau)^n\|L^n\|$, and then
$$1\le(\mu(L)-\tau)\lim_{n\to\infty}\|L^n\|^\frac{1}{n}=\frac{\mu(L)-\tau}{\mu(L)}<1,$$
a contradiction. \par

$(2)$ Let $\tau\in\bR^+$ such that $f^\infty<\mu(L)-\tau$. Then there exists $R_1>0$ such that for every $|u|\ge R_1$ and almost every $t\in [0,1]$
$$
f(t,u)\le (\mu(L)-\tau)|u| .
$$
Also, by $(C_4)$ there exists $\phi_{R_1}\in L^\infty[0,1]$ such that $f(t,u)\le\phi_{R_1}(t)$ for all $u\in[-R_1,R_1]$ and almost every $t\in [0,1]$. Hence,
\begin{equation}\label{supest}
f(t,u)\le(\mu(L)-\tau)|u| +\phi_{R_1}(t)\ \text{for all}\  u\in \bR\ \text{and almost every }\ t\in [0,1].
\end{equation}
Denote by $\Id$  the identity operator and observe that $\Id-(\mu(L)-\tau)L$  is invertible since $(\mu(L)-\tau) L$ has spectral radius less than one. Furthermore, by the Neumann series expression, 
$$
[\Id-(\mu(L)-\tau) L]^{-1}=\sum_{k=0}^\infty[(\mu(L)-\tau) L]^k
$$
therefore,  $[\Id-(\mu(L)-\tau) L]^{-1}$ maps $P$ into $P$, since $L$ does.

Let 
$$
C:=\int_a^b\Phi(s)g(s)\phi_{R_1}(s)ds \ \text{and}\ R_0:= \| [\Id-(\mu(L)-\tau) L]^{-1}C\|.
$$
 Now we prove that for each $R>R_0$, $S u\ne\lambda u$ for all $u\in\partial K_R$ and $\lambda\ge 1$, which implies $ i_{K}(S,K_{R})=1$. Assume otherwise: there exists $u\in\partial K_R$ and $\lambda\ge 1$ such that $\lambda u=Su$.  Taking into account the inequality \eqref{supest}, we have for $t\in [0,1]$
\begin{multline*}
|u(t)|\leq\lambda |u(t)|=  |Su(t)|  =  \left|\int_0^1k_S(t,s)g(s)f(s,u(s))ds\right|\\ \le  \int_0^1|k_S(t,s)|g(s)f(s,u(s))ds 
 \le  (\mu(L)-\tau)\int_0^1|k_S(t,s)|g(s)|u(s)|ds+C
\\ =(\mu(L)-\tau)L |u|(t)+C,
\end{multline*}
which implies 
$$
 [\Id-(\mu(L)-\tau) L]|u|(t)\le C.
 $$
 Since $(\Id-(\mu(L)-\tau) L)^{-1}$ is non-negative, we have
$$
|u|(t)\le [\Id-(\mu(L)-\tau) L]^{-1} C\leq R_0.
$$
Therefore, we have $\|u\|\le R_0<R$, a contradiction.\par
$(3)$ There exists $\rho_0>0$ such that  for all $u\in[0,\rho_0]$ and  all $t\in [a,b]$ we have
$$
f(t,u)\geq \mu(\tilde L)u.
$$
 Let $\rho\in(0,\rho_0]$. Let us prove that
$u\ne Su+\lambda\varphi_1$ for all $u$ in $\partial K_\rho$ and $\lambda\geq 0$,
where $\varphi_1\in K\cap P$ is the eigenfunction of $\tilde L$ with $\|\varphi_1\|=1$ corresponding to the eigenvalue $1/\mu(\tilde L)$. This implies that $ i_{K}(S,K_{\rho})=0$.\par
Assume, on the contrary, that there exist $u\in\partial K_\rho$ and $\lambda\geq 0$ such that $u=Su+\lambda\varphi_1$. \\
We distinguish two cases. Firstly we discuss the case $\lambda>0$. We have, for $t\in [a,b]$, 
\begin{align*}
 u(t)=& \int_0^1k_S(t,s)g(s)f(s,u(s))ds +\lambda \varphi_1(t)\geq\int_a^b k^+_S(t,s)g(s)f(s,u(s))ds+\lambda \varphi_1(t) \\ \geq &\mu(\tilde L) \int_a^b k^+_S(t,s)g(s)u(s)ds+\lambda \varphi_1(t) =\mu(\tilde L) \tilde Lu(t)+\lambda\varphi_1(t).
 \end{align*}
Moreover, we have  $u(t)\ge\lambda\varphi_1(t)$  and then $\tilde Lu(t)\ge\lambda \tilde L\varphi_1(t)\ge \dfrac{\lambda}{\mu(\tilde L)}\varphi_1(t)$ in such a way that we obtain 
$$
u(t)\ge\mu(\tilde L) \tilde Lu(t)+\lambda\varphi_1(t)\ge2\lambda\varphi_1(t),\ \text{   for   } t\in[a,b].
$$
By iteration, we deduce that, for $ t\in[a,b]$, we get
$$
u(t)\ge n\lambda\varphi_1(t)  \text{   for every  } n\in\bN,
$$
 a contradiction because $\|u\|=\rho$.\par
 Now we consider the case $\lambda=0$.  Let $\varepsilon>0$ be such that for all $u\in[0,\rho_0]$ and almost every $t\in [a,b]$ 
we have
\begin{equation*}
f(t,u)\geq (\mu(\tilde L)+\varepsilon)u.
\end{equation*}
We have, for $t\in [a,b]$, 
\begin{equation*}
 u(t)=\int_0^1k_S(t,s)g(s)f(s,u(s))ds \geq\int_a^b k^+_S(t,s)g(s)f(s,u(s))ds \geq(\mu(\tilde L)+\varepsilon) \tilde Lu(t).
\end{equation*}

Since $\tilde L\varphi_1(t)=r(\tilde L)\varphi_1(t)$ for $t\in[0,1]$, we have, for $t\in[a,b]$,
$$
\bar{L} \varphi_1(t)=\tilde L\varphi_1(t)=r(\tilde L)\varphi_1(t),
$$
and we obtain $r(\bar{L})\geq r(\tilde L)$. On the other hand, we have, for $t\in [a,b]$, 
\begin{equation*}
 u(t)\geq(\mu(\tilde L)+\varepsilon) \tilde Lu(t)=(\mu(\tilde L)+\varepsilon) \bar{L} u(t).
\end{equation*}
where $u(t)>0$. Thus, utilizing Theorem~\ref{thmjeff}, we have  $r(\bar{L})\leq \dfrac{1}{\mu(\tilde L)+\varepsilon}$
and therefore $r(\tilde{L})\leq \dfrac{1}{\mu(\tilde L)+\varepsilon}$. This gives $\mu(\tilde L)+\varepsilon\leq \mu(\tilde L)$, a contradiction. \par

$(4)$ Let $R_1>0$ such that 
$$
f(t,u)>\mu(\tilde L)u
$$
 for all $u\ge c R_1$, $c$ as in \eqref{cofks}, and  all $t\in[a,b]$. \\
 Let  $R\geq R_1$. We prove that
$u\ne Su+\lambda\varphi_1$ for all $u$ in $\partial K_R$ and $\lambda\geq 0$, which  implies $ i_{K}(S,K_{R})=0$. \par
Assume now, on the contrary, that there exist $u\in\partial K_R$ and $\lambda\geq 0$ such that $u=Su+\lambda\varphi_1$. Observe that for $u\in\partial K_R$, we have   $u(t)\ge c\|u\|=c R\ge c R_1$ for $t\in[a,b]$. Hence, we have $f(t,u(t))>\mu(\tilde L)u(t)$ for $t\in[a,b]$. This implies, proceeding as in the proof of the statement~$(3)$ for the case $\lambda>0$, that
$$
u(t)\ge\mu(\tilde L) \tilde Lu(t)+\lambda\varphi_1(t)\ge2\lambda\varphi_1(t), \text{ for } t\in[a,b].$$
Then, for $ t\in[a,b]$, we have $u(t)\ge n\lambda\varphi_1(t)$ for every $n\in\bN$, a contradiction because $\|u\|=R$.
The proof in the case $\lambda=0$ is treated as in the proof of the statement~$(3)$.
\end{proof}

The following Theorem, in the line of \cite{jwgi-lms,jwmz-na}, applies the index results in Lemmas \ref{thmind1} and \ref{thmind0} and Theorem \ref{thmindeig}  in order to get some results on existence of multiple nontrivial solutions for the equation \eqref{eqthamm}.

\begin{thm} \label{thmones} 
Assume that conditions $(C1)$-$(C8)$ hold with $\alpha, \beta$ given by positive measures.\\ The integral equation~\eqref{eqthamm} has at least one non-trivial solution in $K$ if one of the following conditions holds. 
\begin{enumerate}
\item[$(H_{1})$] $0\le f^{0}<\mu(L)$ and $\mu(\tilde L)<f_{\infty}\le
\infty$.
\item[$(H_{2})$] $0\le f^{\infty}<\mu(L)$ and $\mu(\tilde L)<f_{0}\le \infty$.
\end{enumerate}
The integral equation~\eqref{eqthamm} has at least two non-trivial solutions in $K$ if one of
the following conditions holds.
\begin{enumerate}
\item[$(Z_{1})$] $0\leq f^{0}<\mu(L)$,  $f_{{\rho},\rho/c}>M_S(a,b)$ for
some $\rho>0$, and $0\le f^{\infty}<\mu(L)$.
\item[$(Z_{2})$] $\mu(\tilde L)<f_{0}\le \infty$, $f^{-\rho,\rho}< m_S$ for some
$\rho>0$, and $\mu(\tilde L)<f_{\infty}\leq \infty$.
\end{enumerate}
The integral equation~\eqref{eqthamm} has at least three non-trivial solutions in $K$ if one of
the following conditions holds.
\begin{enumerate}
\item[$(T_{1})$] There exist  $0<{\rho}_{1} < {\rho}_{2} < \infty$, such
that
\begin{equation*}\;\;\mu(\tilde L)<f_{0}\leq \infty,  \;\; f^{-\rho_1,\rho_{1}}<m_S,
\;\; f_{\rho_{2},\rho_{2}/c}>M_S(a,b), \;\; 0 \leq f^{\infty}
<\mu(L).
\end{equation*}
\item[$(T_{2})$] There exist  $0<{\rho}_{1} < c{\rho}_{2} < \infty$, such
that
\begin{equation*}\;\; 0 \leq f^{0}<\mu(L),
\;\;f_{\rho_{1},\rho_{1}/c}>M_S(a,b),
\;\; f^{-\rho_2,\rho_{2}}<m_S, \;\; \mu(\tilde L)_1<f_{\infty}\le
\infty.
\end{equation*}
\end{enumerate}
\end{thm}
It is possible to give criteria for the existence of an arbitrary
number of nontrivial solutions by extending the list of conditions.
We omit the routine statement of such results.

The following Lemma sheds some light on the relation between some of these constants.

\begin{lem}
The following relations hold
$$M_S(a,b)\geq\mu(\tilde L)\geq\mu( L)\geq m_S.$$
\end{lem}
\begin{proof}
The fact that $\mu( L)\geq m_S$ essentially follows from Theorem 2.8 of \cite{jwkleig}. The comment that follows after Theorem 3.4 of \cite{jwkleig} also applies in our case, giving $\mu(\tilde L)\geq\mu( L)$.

We now prove $M_S(a,b)\geq \mu(\tilde{L} )$.
Let $\varphi \in P\cap K$ be a corresponding eigenfunction of norm $1$ of $1/\mu (\tilde{L})$ for the operator $\tilde{L}$, that is $\varphi=\mu(\tilde{L})\tilde{L} (\varphi)$ and $\Vert \varphi\Vert=1$.
Then, for $t\in [a,b]$ we have
\begin{equation*}
\varphi(t)= \mu(\tilde L)\int_a^bk_S(t,s)g(s)\varphi(s)ds\geq\mu(\tilde L)\min_{t\in [a,b]}\varphi(t)\int_a^bk_S(t,s)g(s)ds
.\end{equation*}
Taking the infimum over $[a,b]$, we obtain 
\begin{equation*}
\min_{t\in [a,b]}\varphi(t)\geq  \mu(\tilde{L} ) \min_{t\in [a,b]}\varphi(t)/M_S(a,b),
\end{equation*}
that is $M_S(a,b)\geq \mu(\tilde{L} )$.
\end{proof}
In order to present an index zero result of a different nature, we introduce the following operator
$$
L_+\,u(t):=  \int_0^1 k^+_{S}(t,s)g(s)u(s)\,ds,
$$
for which a result similar to Theorems~\ref{lcomp} and~\ref{specrad} holds.

In the next Theorem we use the following notation, with $c$ as in \eqref{cofks},
\begin{equation}\label{fmd}
\tilde f_{0}=\varliminf_{u\to 0}\frac{\mathrm{ess}\inf\limits_{t \in [0,1]}f(t,u)}{|u|},
\quad
\tilde c:=\frac{1}{c}\sup_{t\in [0,1]}\frac{\int_0^1k_S^-(t,s)g(s)ds}{\int_a^bk_S^+(t,s)g(s)ds}.
\end{equation}
 
\begin{thm}\label{thmindnew}
 If $\mu(L_+)< \tilde f_0-\tilde c\, f^0$,  then there exists
$\rho_{0}>0$ such that for each $\rho\in (0,\rho_{0}]$, if $u \neq
Tu$ for $u \in \partial{K}_{\rho}$, it is satisfied that
$
i_{K}(T,K_{\rho})=0.
$
\end{thm}
\begin{proof}
Firstly, since $u\in K$ we have, for $t\in[0,1]$,
\begin{align*}
\int_0^1k_S^-(t,s)g(s)|u(s)|ds\le  &\int_0^1k_S^-(t,s)g(s)\|u\|ds\le\tilde c\int_a^bk_S^+(t,s)g(s)c\|u\|ds \\ \le & \tilde c\int_a^bk_S^+(t,s)g(s)|u(s)|ds\leq\tilde c\, L_+|u|(t).
\end{align*}

 Observe that the hypothesis $\mu(L_+)< \tilde f_0-\tilde c\, f^0$ implies $\tilde f_0,f^0<\infty$. Let $\rho_0>0$ such that 
\begin{equation*}
f(t,u)\geq (\mu(L_+)+\tilde c\, f^0)|u|   \, \text{    and     } f(t,u)\le (f^0+\mu(L_+)/2)|u|
\end{equation*}
for all $u\in[-\rho_0,\rho_0]$ and almost all $t\in [0,1]$. \\
Let $\rho\leq \rho_0$. We will prove that
$u\ne Su+\lambda\varphi_+$ for all $u$ in $\partial K_\rho$ and $\lambda>0$ where $\varphi_+\in K$ is an eigenfunction of $L_+$ related to the eigenvalue $1/\mu(L_+)$ such that $\|\varphi_+\|=1$.
\par
Assume now, on the contrary, that there exist $u\in\partial K_\rho$ and $\lambda>0$ such that $u(t)=Su(t)+\lambda\varphi_+(t)$ for all $t\in [0,1]$. Hence, we have
$$
u(t)=-\int_0^1 k^-_{S}(t,s)g(s)f(s,u(s))\,ds+\int_0^1 k^+_{S}(t,s)g(s)f(s,u(s))\,ds+\lambda\varphi_+(t).
$$
On one hand, we have
\begin{align*}
 u(t)+\int_0^1k_S^-(t,s)g(s)f(s,u(s))ds\le &|u(t)|+[f^ 0+\tfrac{1}{2}\mu(L_+)] \int_0^1k_S^-(t,s)g(s)|u(s)|ds\\ &\le  |u(t)|+\tilde c[f^ 0+\tfrac{1}{2}\mu(L_+)] L_+|u|(t).
\end{align*}
On the other hand, we have
$$
\int_0^1 k^+_{S}(t,s)g(s)f(s,u(s))\,ds+\lambda\varphi_+(t)\ge (\mu(L_+)+\tilde c\, f^ 0)L_+|u|(t)+\lambda\varphi_+(t).
$$
Therefore, we obtain
$$
(\mu(L_+)+\tilde c\, f^ 0)L_+|u|(t)+\lambda\varphi_+\le |u(t)|+\tilde c\,[f^ 0+\tfrac{1}{2}\mu(L_+)] L_+|u|(t),
$$
or, equivalently,
$$
\tfrac{1}{2}\mu(L_+) L_+|u|(t)+\lambda\varphi_+(t)\le|u(t)|.
$$
Hence we get
 $$\lambda\varphi_+(t)\le|u(t)|.$$
Reasoning as in the proof of $(3)$ of Theorem \ref{thmindeig}, we obtain
$$
|u(t)|\ge\lambda\tfrac{1}{2}\mu(L_+) L_+\varphi_+(t)+\lambda\varphi_+(t)=\tfrac{3}{2}\lambda\varphi_+(t).
$$
By induction we deduce that $|u(t)|\ge (\tfrac{n}{2}+1)\lambda\varphi_+(t)$ for every $n\in\bN$, a contradiction since $\|u\|=\rho$.
\end{proof}

As in the Theorem \eqref{thmones}, results on existence of multiple nontrivial solutions can be established. We omit the  statement of such results.

\begin{rem} 
The hypothesis in the Theorem \ref{thmindnew} imply that $\tilde c\in(0,1)$. Also, if $\tilde f_0=f^0=f_0$ then the hypothesis in Theorem \ref{thmindnew} is equivalent to $\mu(L^+)/(1-\tilde c)<\tilde f_0<\infty$. Furthermore, if $[a,b]=[0,1]$, then $L=L^+=\tilde L$ and the growth condition becomes $\mu(L)<\tilde f_0<\infty$, which is condition (3) in Theorem~\ref{thmindeig} for $f_0<\infty$.
\end{rem}

\section{Study of the Green's functions of the BVPs \eqref{shiftint1}-\eqref{shiftint2}}\label{Greenstudy}
In this Section we study the properties of the Green's function of the BVP
\begin{equation*}
 \e u''(t)+ \omega^2u(t)=y(t),\quad u'(0)=u'(1)=0,
 \end{equation*}
where $y\in L^1[0,1]$, $\e=\pm 1$ and $\omega\in\bR^+$.
We discuss separately two cases.
\subsection{CASE $\e=-1$} \par 
The Green's function $k$ of BVP
 $$
 - u''(t)+ \omega^2u(t)=y(t),\quad u'(0)=u'(1)=0,
 $$
  is given by (see for instance \cite{Wan} or \cite{Yao}),
$$
\omega\sinh \omega\, k(t,s):=\begin{cases}
\cosh\omega (1-t)\cosh\omega s, & 0\le s\le t\le 1, \\ \cosh\omega (1-s)\cosh\omega t, & 0\le t\le s\le 1.
\end{cases}
$$
Note that $k$ is continuous, positive and satisfies some symmetry properties such as 
$$k(t,s)=k(s,t)=k(1-t,1-s).$$ 
Observe that $\dfrac{\partial k}{\partial t}(t,s)<0$ for $s<t$ and $\dfrac{\partial k}{\partial t}(t,s)>0$ for $s>t$. Therefore we choose 
$$\Phi(s):=\sup_{t\in [0,1]} k(t,s)= k(s,s).$$

For a fixed $[a,b]\subset [0,1]$ we have
$$
c(a,b):=\min_{t\in[a,b]}\min_{s\in[0,1]}\frac{k(t,s)}{\Phi(s)}=\frac{\min\left\{\cosh\omega a,\cosh \omega (1-b)\right\}}{\cosh \omega}.
$$
The choice of $g\equiv 1$ gives
\begin{equation*}
\frac{1}{m}=\sup_{t\in [0,1]}\int_0^1k(t,s)\,ds,
\end{equation*}
and, by direct calculation, we obtain that $m=\omega^2$.\par
The constant $M$ can be computed as follows
\begin{align*}
\frac{1}{M(a,b)}:=&\inf_{t\in[a,b]}\int_a^bk(t,s)ds\\
=&\frac{1}{\omega^2}-\sup_{t\in[a,b]}\frac{\sinh \omega a\cosh\omega(1-t)+\sinh\omega(1-b)\cosh\omega t}{\omega^2\sinh\omega}.
\end{align*}
Let $\xi_1(t):=\sinh \omega a\cosh\omega(1-t)+\sinh\omega(1-b)\cosh\omega t$. Then we have $\xi_1''(t)=\omega^2\xi(t)\ge0$. Therefore the supremum of $\xi_1$ must be attained in one of the endpoints of the interval $[a,b]$. Thus we have
$$
\frac{1}{M(a,b)}=\frac{1}{\omega^2}-\frac{\max\{\xi_1(a),\xi_1(b)\}}{\omega^2\sinh\omega}.
$$
Note that
$$
\xi_1(b)-\xi_1(a)=-2\sinh^2\(\frac{b-a}{2}\omega\)\sinh\omega(a+b-1),
$$
and therefore, $\xi_1(b)\ge\xi_1(a)$ if and only if $a+b\le 1$. Hence we obtain
$$
\frac{1}{M(a,b)}=\frac{1}{\omega^2}-\frac{1}{\omega^2\sinh\omega}\begin{cases}
\sinh \omega a\cosh\omega(1-b)+\sinh\omega(1-b)\cosh\omega b,\, a+b\le 1, \\ \sinh \omega a\cosh\omega(1-a)+\sinh\omega(1-b)\cosh\omega a,\, a+b>1.
\end{cases}
$$
\subsection{CASE $\epsilon=1$}\par The Green's function $k$ of the BVP
 $$
  u''(t)+ \omega^2u(t)=y(t),\quad u'(0)=u'(1)=0,
 $$
  is given by  
$$\omega\sin \omega\, k(t,s):=\begin{cases}
\cos\omega (1-t)\cos\omega s, & 0\le s\le t\le 1, \\ \cos\omega (1-s)\cos\omega t, & 0\le t\le s\le 1.
\end{cases}
$$

In the following Lemma we describe the sign properties of this Green's function with respect to the parameter $\omega$. A similar study has been done, for different BVPs, in \cite[Theorem~4.3]{alb-adr} and \cite[Lemma~5.2]{ac-gi-at-bvp}. The proof is straightforward and is omitted.

\begin{lem}\label{lemsignneu}
We have the following.
\begin{enumerate}
\item $k$ is positive for $\omega\in(0,\pi/2)$.
\item $k$ is positive for $\omega=\pi/2$ except at the points $(0,0)$ and $(1,1)$ where it is zero.
\item $k$ is positive on the strip $(1-\pi/(2\omega),\pi/(2\omega))\times [0,1]$ if $\omega\in(\pi/2,\pi)$.
\item if $\omega>\pi$, there is no strip of the form $(a,b)\times [0,1]$ where $k$ is positive.
\end{enumerate}
\end{lem}

Consider $\omega\in(0,\pi)$. Fix $s\in [0,1]$ and note that $\dfrac{\partial k}{\partial t}(t,s)$ never changes sign for $t\in[0, s)$ nor for $t\in(s,1]$.
Thus we can take
\begin{align*}
\Phi(s): & =\sup_{t\in [0,1]} |k(t,s)|=\max\{|k(0,s)|,|k(1,s)|,|k(s,s)|\} \\
 &=\frac{\max\{|\cos\omega(1-s)|,|\cos\omega s|,|\cos\omega s\cos\omega(1-s)|\}}{\omega\sin \omega}\\
 & =\frac{\max\{\cos\omega(1-s),\cos\omega s\}}{\omega\sin \omega}.
\end{align*}
The last equality holds because $\cos(\omega s)\ge-\cos\omega(1-s)\ge 0$ for $s\le1-\pi/(2\omega)$ and $\cos(1-\omega s)\ge-\cos\omega s\ge 0$ for $s\ge\pi/(2\omega)$.\newline
On the other hand, for $[a,b]\subset (\max\{0,1-\pi/(2\omega)\},\min\{1,\pi/(2\omega)\})$, we have
$$
\inf_{t\in[a,b]} k(t,s)=\begin{cases} \min\left\{k(a,s),k(b,s)\right\}, & s\in[0,1]\backslash[a,b], \\ \min\left\{k(a,s),k(s,s),k(b,s)\right\}, & s\in[a,b].\end{cases}
$$
 Now, we study the three intervals $[0,a)$, $[a,b]$ and $(b,1]$ separately.
 
 If  $s\in[0,a)$, we have
\begin{align*}
& \inf_{s\in[0,a)}\frac{\min\left\{k(a,s),k(b,s)\right\}}{\Phi(s)} \\ 
= & \inf_{s\in[0,a)}\frac{\min\left\{\cos \omega(1-a)\cos\omega s,\cos \omega (1-b)\cos\omega s\right\}}{\max\{\cos\omega(1-s),\cos\omega s\}} \\  
= & \inf_{s\in[0,a)}\min\left\{\cos \omega (1-a),\cos \omega (1-b),\frac{\cos\omega(1-a)\cos\omega s}{\cos\omega(1-s)},\frac{\cos \omega (1-b)\cos\omega s}{\cos\omega(1-s)}\right\} \\  = & \min\left\{\cos \omega (1-a),\cos \omega (1-b),\cos \omega a,\cos \omega (1-b)\frac{\cos\omega a}{\cos\omega(1-a)}\right\}\\
= & \min\left\{\cos \omega (1-a),\cos \omega (1-b),\cos \omega a\right\},
\end{align*}
where these equalities hold because $\dfrac{\cos\omega s}{\cos\omega(1-s)}$ is a decreasing function for $s\in [\max\{0,1-\pi/(2\omega)\},1]$ and the function cosine is decreasing in $[0,\pi]$. 
 
 If  $s\in[a,b]$, we have
\begin{align*}
& \inf_{s\in[a,b]}\frac{\min\left\{k(a,s),k(s,s),k(b,s)\right\}}{\Phi(s)} \\ 
= & \inf_{s\in[a,b]}\frac{\min\left\{\cos \omega a\cos\omega (1- s),\cos\omega s\cos\omega(1-s),\cos \omega (1-b)\cos\omega s\right\}}{\max\{\cos\omega(1-s),\cos\omega s\}} \\  
= & \inf_{s\in[a,b]}\min\Bigl\{\cos \omega a,\cos \omega (1-b),\cos\omega s,\cos\omega(1-s),\cos \omega a\frac{\cos\omega (1-s)}{\cos\omega s},\\
&\cos \omega (1-b)\frac{\cos\omega s}{\cos\omega(1-s)}\Bigr\} \\  
= & \min\left\{\cos \omega a,\cos \omega (1-b),\cos\omega b,\cos\omega(1-a)\right\}.
\end{align*}
 If  $s\in (b,1])$, we have
\begin{align*}
& \inf_{s\in(b,1]}\frac{\min\left\{k(a,s),k(b,s)\right\}}{\Phi(s)} \\ = & \inf_{s\in(b,1]}\frac{\min\left\{\cos \omega a\cos\omega (1- s),\cos \omega b\cos\omega (1- s)\right\}}{\max\{\cos\omega(1-s),\cos\omega s\}} \\  = & \inf_{s\in(b,1]}\min\left\{\cos \omega a,\cos \omega b,\cos \omega a\frac{\cos\omega (1-s)}{\cos\omega s},\cos \omega b\frac{\cos\omega (1-s)}{\cos\omega s}\right\} \\  = & \min\left\{\cos \omega a,\cos \omega b,\cos \omega a\frac{\cos\omega (1-b)}{\cos\omega b},\cos \omega (1-b)\right\}\\  = & \min\left\{\cos \omega a,\cos \omega b,\cos \omega (1-b)\right\}.
\end{align*}
Therefore, taking into account these three infima, we obtain that
$$
c(a,b):=\inf_{s\in [0,1]}\frac{\inf_{t\in[a,b]}k(t,s)}{\Phi(s)}=\min\left\{\cos \omega a,\cos \omega (1-a),\cos \omega b,\cos \omega (1-b)\right\}.
$$
In order to compute the constant $m$ we use Lemma \ref{lemsignneu} and the fact that $k(t,s)=k(s,t)$ for all $t,s\in[0,1]$.\\
 If $\omega\in(0,\pi/2)$, the function $k$ is positive and therefore 
 $$
 m=\omega^2.
 $$
 If $\omega\in[\pi/2,\pi)$, we have
$$
\zeta(t):=\int_0^1k^+(t,s)\,ds=\begin{cases}\int_{1-\frac{\pi}{2\omega}}^1k(t,s)\,ds=\dfrac{1}{\omega^2}\dfrac{\cos\omega t}{\sin \omega}, \, t\in[0,1-\frac{\pi}{2\omega}), \\
\dfrac{1}{\omega^2}, \, t\in[1-\frac{\pi}{2\omega},\frac{\pi}{2\omega}], \\
\int_0^{\frac{\pi}{2\omega}}k(t,s)\,ds=\dfrac{1}{\omega^2}\dfrac{\cos\omega(1- t)}{\sin \omega}, \, t\in(\frac{\pi}{2\omega},1]. 
\end{cases}$$
Since
$$
0<\frac{1}{\omega^2}=\int_0^1k(t,s)\,ds=\int_0^1k^+(t,s)\,ds-\int_0^1k^-(t,s)\,ds,
$$
we obtain that  $\int_0^1k^+(t,s)\,ds>\int_0^1k^-(t,s)\,ds$, in such a way that 
$$
m=1/\max_{t\in[0,1]}\zeta(t)=\omega^2\,\sin\omega.
$$
 Also we have
 $$
 \frac{1}{M(a,b)}=\frac{1}{\omega^2}-\sup_{t\in[a,b]}\frac{\cos\omega(1-t)\sin\omega a+\cos\omega t\sin\omega(1-b)}{\omega^2\,\sin\omega}.
 $$
Denote by
 $$\xi_3(t):=\cos\omega(1-t)\sin\omega a+\cos\omega t\sin\omega(1-b),
 $$ 
 and observe that 
 $$
 \xi_3(t)=\omega^2\,\sin\omega\(\int_0^1k(t,s)ds-\int_a^bk(t,s)ds\),
 $$ 
 and therefore we have $\xi_3(t)\geq 0$ for $t\in[a,b]$. Then, we have
$$
\xi_3'(a)\xi_3'(b)=-4\omega^2\cos\[\frac{\omega}{2}(2-a+b)\]\cos\[\frac{\omega}{2}(a+b)\]\sin^2\[\frac{\omega}{2}(a-b)\]\sin\omega(1-b)\sin\omega a.
$$
Now, $\xi_3'(a)\xi_3'(b)<0$ if and only if $2-\pi/\omega<a+b<\pi/\omega$, which is always satisfied for $[a,b]\subset\(1-\pi/(2\omega),\pi/(2\omega)\)$. In such a case, $\xi_3$ has a maximum in $[a,b]$, precisely at the unique point $t_0$ satisfying
$$
\sin\omega t_0=\frac{\sin\omega\sin\omega a}{\cos\omega\sin\omega\a+\sin\omega(1-b)}\cos\omega t_0.
$$
Thus we obtain
\begin{align*}
\xi_3(t_0)  =&\cos\omega\cos\omega b\cos\omega t_0+\cos\omega\sin\omega a\cos\omega t_0-\cos\omega\sin\omega b\cos\omega t_0\\
&+\sin\omega\sin\omega a\sin\omega t_0 \\ 
 =&\(\cos\omega\cos\omega b+\cos\omega\sin\omega a-\cos\omega\sin\omega b+\frac{(\sin\omega\sin\omega a)^2}{\cos\omega\sin\omega\a+\sin\omega(1-b)}\)\cos\omega t_0 \\ 
 =&\frac{\left|\cos\omega\cos\omega b+\cos\omega\sin\omega a-\cos\omega\sin\omega b+\dfrac{(\sin\omega\sin\omega a)^2}{\cos\omega\sin\omega\a+\sin\omega(1-b)}\right|}{\sqrt{\(\dfrac{\sin\omega\sin\omega a}{\cos\omega\sin\omega\a+\sin\omega(1-b)}\)^2+1}}
  \end{align*}
\begin{rem}\label{remint}
 In the particular case $a+b=1$, we have $\xi_3(t)=\sin\omega a[\cos\omega(1-t)+\cos\omega t]$. In this case, observe that $\xi_3(t)=\xi_3(1-t)$ and recall that $\xi_3''(t)=-\omega^2\xi_3(t)\ge0$ ($\xi_3$ is not constantly zero in any open subinterval). Therefore the maximum is reached at the only point where $t=1-t$, that is, $t=1/2$. Hence we obtain
$$
\frac{1}{M(a,b)}=\frac{1-2\cos\dfrac{\omega}{2}\sin\omega a}{\omega^2\sin\omega}.
$$
\end{rem}\label{cintre}
\begin{rem}
The constants $m$, $M(a,b)$, $c(a,b)$ and the function $\Phi$ improve and complement some of the ones used in \cite{Sun2, Sun3, Sun, Wang1, Wan, Yao2, Yao}.
\end{rem}
\section{Examples}\label{secex}
In this Section we present some examples in order to illustrate some of the constants that occur in our theory and the applicability of our theoretical results. Note that the constants that occur are rounded to the third decimal place unless exact.

In the first example we study the existence of multiple nontrivial solutions of a (local) Neumann BVP.
\begin{exa}
Consider the BVP
\begin{equation} \label{eqexasing} 
u''(t)+\(\frac{7\pi}{12}\)^2  u(t) =\frac{\tau_1u^2(t)}{1+t^2} e^{-\tau_2 |u(t)|},\ t\in [0,1],\ u'(0)=u'(1)=0,
\end{equation}
where $\tau_1,\tau_2>0$.

In this case $\omega=\frac{7\pi}{12}$ and, by Lemma~\ref{lemsignneu}, the Green's function is positive on the strip $(1/7,6/7)\times [0,1]$. We illustrate the Remark~\ref{remint} by choosing $[1/4,3/4]\subset (1/7,6/7)$ and we prove, by means of Theorem~\ref{thmones}, the existence of two nontrivial solutions of the BVP~\eqref{eqexasing} which are (strictly) positive on the interval $[1/4,3/4]$. 

In order to do this, note that in our case we have $f(t,u)=\dfrac{\tau_1u^2}{1+t^2} e^{-\tau_2 |u|}$ and $f^0=f^\infty=0$. Furthermore, using the results in the previous Section, we have  
\begin{equation}\label{cex1}
c(1/4,3/4)  =\cos\(\frac{7\pi}{16}\)=\frac{1}{2}\sqrt{2-\sqrt{2+\sqrt{2}}}= 0.195,
\end{equation}
and 
$$
M=M(1/4,3/4) = 7.029.
$$
Henceforth we work in the cone 
$$K=\{u\in C[0,1]: \min_{t \in [1/4,3/4]}u(t)\geq c \|u\|\},$$ 
with $c$ given by~\eqref{cex1}.

We set
$$
\hat{f}_0:=2\frac{c-1}{\ln c}c^{\frac{c}{c-1}}M=10.289.
$$
We now prove that if $\tau_1/\tau_2>\hat{f}_0$, then the condition $(Z_1)$ is satisfied.
Let 
$$
\hat{f}(u):=\inf_{t\in[0,1]}\frac{\tau_1u^2}{1+t^2}e^{-\tau_2 u}=\dfrac{1}{2}\tau_1u^2e^{-\tau_2 u},\ u\in[0,+\infty).
$$

Note that $\hat{f}'$ only vanishes at $0$ and $2/\tau_2$, $\hat{f}$ is strictly increasing in the interval $(0, \frac{2}{\tau_2})$ and  is strictly decreasing in the interval $(\frac{2}{\tau_2}, +\infty)$. Thus $\hat{f}$ assumes the maximum in the unique point $2/\tau_2$ and, since $\hat{f}(0)=0$ and $\lim_{x\to +\infty}\hat{f}(x)=0$,   the inverse image by $\hat{f}$ of any strictly positive real number different to $\hat{f}(\frac{2}{\tau_2})$ has either $2$ or no points. 
Let for $x\in [0,+\infty)$
$$
l(x):=\hat{f}(x)-\hat{f}(x/c).
$$
Take $\varepsilon\in (0,\frac{2c}{\tau_2})$ and note that $l(\varepsilon)<0$ in view of the strict monotonicity of $\hat{f}$. Moreover, if $\eta>\frac{2}{\tau_2}$, then $l(\eta)>0$.
Since the function $l$ is continuous, there exists a point $\bar{x}\in (\varepsilon,\eta)$ such that $l(\bar{x})=0$, that is, $\hat{f}(\bar{x})$=$\hat{f}(\bar{x}/c)= p$. 
From  the type of monotonicity of $f$, for $x\in [\bar{x},\bar{x}/c]$ we have $p\leq\hat{f}(x)$. Hence we have
$$
\hat{f}(\bar{x})=\hat{f}(\bar{x}/c)\Ra \bar{x}=e^{\tau_2(\bar{x}/c-\bar{x})}\bar{x}/c\Ra \bar{x}=\frac{2c\ln c}{\tau_2(c-1)},\bar{x}/c=\frac{2\ln c}{\tau_2(c-1)}.
$$
Thus, if we impose  $p>M\,\bar{x}$, we obtain
$$
M\frac{2c\ln c}{\tau_2(c-1)}=M\,\bar{x}<\hat{f}(\bar{x})=\hat{f}(\bar{x}/c)=\tau_1\(\frac{2\ln c}{\tau_2(c-1)}\)^2c^{-\frac{2}{c-1}},
$$
that is, $\tau_1/\tau_2>\hat{f}_0$.
\end{exa}
We now present an example for a BVP subject to two nonlocal BCs.
\begin{exa}
Consider the BVP
\begin{equation}\label{exan3}
\begin{aligned}
 u''(t)+\omega^2u(t)& =e^{-|u(t)|},\ t\in [0,1],\\
u'(0)&=u(0)+u(1),\\
u'(1)&=\int_0^{1}u(t)\sin\pi t dt,
\end{aligned}
\end{equation}
where $\omega\in(\pi/2,\pi)$.
We rewrite the BVP~\eqref{exan3} in the integral form
\begin{equation*}
Tu(t)={\gamma}(t)\alpha[u]+{\delta(t)}{\beta}[u] +\int_0^1 k(t,s)f(s,u(s))\,ds,
\end{equation*}
where
\begin{align*}
\c(t)=&\cos\omega(1- t)/(\omega\sin\omega),\,\,\,
\d(t)=\cos(\omega t)/(\omega\sin\omega),\\
\a[u]= & u(0)+u(1),\,\,\,
\b[u]= \int_0^{1}u(t)\sin\pi t dt.
\end{align*}
In order to verify condition $(S_1)$ of Theorem~\ref{thmcasesS}, we take $[a,b]\subset (1-\pi/(2\omega),\pi/(2\omega))$ and let $f(u)=e^{-|u|}$. 

Note that the condition $f^\infty=0$ implies that the condition $(I_\rho^1)$ is satisfied for $\rho$ sufficiently large (hence $i_K(T,K_R)=1$ for $R$ big enough).

Now it is left to prove that $i_K(T,V_\rho)=0$ for $\rho$ small enough (condition $(I_\rho^0)$).\par
We have
\begin{align*}
\a[\c]=&\a[\d]=\sqrt{2}\frac{\sin\left(\frac{\pi }{4}+\omega \right)}{\omega\sin\omega},\ 
\b[\c]=  \b[\d]=\frac{\pi  \cot\left(\frac{\omega }{2}\right)}{\pi ^2 \omega -\omega ^3},\\
D:=D(\omega)=& \frac{\left(\pi ^2 \omega -\omega ^3\right)\sin(\omega /2)-\left(\pi +\pi ^2-\omega ^2\right) \cos\left(\frac{\omega }{2}\right)}{\left(\pi ^2 \omega -\omega ^3\right)\sin(\omega /2)},\\
\mathcal{K}_A(s)= &\frac{\cos\omega s+\cos(\omega[ 1-s])}{\omega\sin\omega },\\
\mathcal{K}_B(s)= & \frac{\pi  \cos\omega s \cot(\omega/2)-\omega\sin\pi s+\pi\sin\omega s}{\pi ^2 \omega -\omega ^3}.
\end{align*}

Observe that $\a[\c],\a[\d],\b[\c],\b[\d],\mathcal{K}_A(s),\mathcal{K}_B(s)\ge 0$ and $\a[\c],\b[\d]<1$ for $\omega\in(\pi/2,\pi)$.\\ Also, we have $D(\omega)>0$ for $\omega\in(\pi/2,\pi)$. In fact, $D(\omega)$ is a strictly increasing function (since $D'(\omega)>0$ for $\omega\in(0,\pi)$),  $\lim_{\omega\rightarrow 0^+}D(\omega)=-\infty$ and $D(\pi)=1-\frac{1}{4 \pi }>0$, so there is a unique zero $\omega_0$ of $D$ in $(0,\pi)$ but $\omega_0=1.507<\pi/2.$\par
Now, $\c$ is increasing and $\d$ is decreasing, therefore $c_2=\c(a)/\c(1)=\cos(\omega[1-a])$, $c_3=\d(b)/\d(0)=\cos\omega b$. On the other hand, we have
\begin{align*}
 f_{\rho,\rho/c}= & f(\rho/c)/(\rho/c)=e^{-\rho/c}c/\rho,\\
 c(a,b)= & \min\{\cos \omega a,\cos\omega(1-a),\cos\omega b,\cos\omega(1-b)\},\\
 \int_a^b\mathcal{K}_A(s)ds= & \frac{\sin\omega b-\sin\omega a+\sin\omega(1-a)-\sin\omega(1-b)}{\omega ^2\sin\omega},\\
 \int_a^b\mathcal{K}_B(s)ds= &\frac{\pi ^2 \left(\cot \left(\frac{\omega }{2}\right) (\sin (b \omega )-\sin (a \omega ))+\cos (a \omega )-\cos (b \omega )\right)}{\omega ^2 \left(\pi ^3-\pi  \omega ^2\right)}\\
 &- \frac{\omega ^2( \cos (\pi  a)+ \cos (\pi  b))}{\omega ^2 \left(\pi ^3-\pi  \omega ^2\right)}.
 \end{align*}
 Taking $a+b=1$, we obtain
 \begin{align*}
  \int_a^b\mathcal{K}_A(s)ds= & \frac{2 \csc \left(\frac{\omega }{2}\right) \sin \left(\frac{1}{2} (\omega -2 a \omega )\right)}{\omega ^2},\\
  \int_a^b\mathcal{K}_B(s)ds= &-\frac{2 \left(\omega ^2 \cos (\pi  a)-\pi ^2 \cos (a \omega )+\pi ^2 \cot \left(\frac{\omega }{2}\right) \sin (a \omega )\right)}{\omega ^2 \left(\pi ^3-\pi  \omega ^2\right)},\\
  c=\cos\omega a.
  \end{align*}
Condition $(I_\rho^0)$ is equivalent to
  $$ f_{\rho,\rho/c}\cdot \inf_{t\in[a,b]}\left\{q(t,\omega, a)+\int_a^bk(t,s)ds\right\}>1,$$
  where
  \begin{multline*} q(t,\omega,a)=\\  \frac{2 \csc (\omega ) \left(\pi  \csc \left(\frac{\omega }{2}\right) \sin \left(\frac{1}{2} (\omega -2 a \omega )\right) (\pi  \cos (t \omega )+(\pi -\omega ) (\omega +\pi ) \cos (\omega -t \omega ))\right)}{\pi  \omega ^2 \left((\pi -\omega ) \omega  (\omega +\pi )-\left(-\omega ^2+\pi ^2+\pi \right) \cot \left(\frac{\omega }{2}\right)\right)}\\
 -\frac{2\omega \csc (\omega )   \cos (\pi  a) (\sin (t \omega )-\sin (\omega -t \omega )+\omega  \cos (t \omega ))}{\pi  \omega ^2 \left((\pi -\omega ) \omega  (\omega +\pi )-\left(-\omega ^2+\pi ^2+\pi \right) \cot \left(\frac{\omega }{2}\right)\right)}.
  \end{multline*} 
 Using Remark~\ref{easyidx0}, it is enough to check
 $$ f_{\rho,\rho/c}\cdot\Bigl(\inf_{t\in[a,b]}q(t,\omega,a)+\frac{1}{M(a,b)}\Bigr)>1.$$
 It can be checked numerically that $\inf_{t\in[a,b]}q(t,\omega,a)=q(a,\omega,a)$.
 Hence, we need
  $$ \frac{e^{-\rho/\cos\omega a}\cos\omega a}{\rho}\Bigl( q(a,\omega,a)+\frac{1-2\cos\frac{\omega}{2}\sin\omega a}{\omega^2\,\sin\omega}\Bigr)>1.$$
Since $\lim_{\rho\rightarrow 0}e^{-\rho/\cos\omega a}/\rho=+\infty$, the inequality is satisfied for $\rho$ small enough and, hence, we have proved that the BVP~\eqref{exan3} has at least a non trivial solution in the cone $K$.
 \end{exa} 
We now study an example that occurs in an earlier article by Bonanno and Pizzimenti~\cite{BonPiz}.
 \begin{exa}
Consider the BVP
 \begin{equation}\label{lastne} 
 -u''(t)+u(t)=\l te^{u(t)}, t\in [0,1],\quad u'(0)=u'(1)=0.
 \end{equation}
 In \cite{BonPiz} the authors establish the existence of at least one positive solution such that $\|u\|<2$ for $\l\in(0,2e^{-2})$. 
 
 The BVP \eqref{lastne} is equivalent to the following integral problem
 \begin{equation*}
 u(t)=\int_0^1k(t,s)g(s)f(u(s))ds, 
 \end{equation*}
 where
 $$
 g(s)=s,\ f(u)=\l e^u
 $$
 and
 $$
k(t,s):=\frac{1}{ \sinh(1)}\begin{cases}
 \cosh (1-t)\cosh s, & 0\le s\le t\le 1, \\ \cosh (1-s)\cosh t, & 0\le t\le s\le 1.
 \end{cases}
 $$
The kernel $k$ is positive, by the results Section~5 satisfies $(C_1)$-$(C_8)$ with $[a,b]=[0,1]$. Thus we work in the cone 
$$K=\{u\in C[0,1]: \min_{t \in [0,1]}u(t)\geq c \|u\|\},$$ 
where $$ c=c(0,1)=  1/\cosh 1=0.648.$$

 We can compute the following constants
 \begin{align*}
 m=& \frac{e+1}{2}=1.859,\\
M(0,1)= & \frac{e+1}{e-1}=2.163,\\
 f^{0,\rho}= & f_{\rho,\rho/c}=\l e^\rho/\rho.
 \end{align*}

Taking $\rho_2=2$ we have $(I^1_{\rho_2})$ is satisfied for $\l<(e+1)e^{-2}$, and taking $0<\rho_1<c/2$ we have $(I^0_{\rho_1})$ for $\l>[(e+1)/(e-1)]\rho_1 e^{-\rho_1}$.

Hence, the condition $(S_1)$ of Theorem~\ref{thmcasesS} is satisfied whenever  
$$
 \l\in\(0,\frac{e+1}{e^{2}}\)
 \supset(0,2e^{-2}).
 $$

Furthermore, reasoning as in \cite{gippmt}, when $\lambda=\dfrac{1}{4}$ the choice of $\rho_2=0.16$ and $\rho_1=0.1$ gives the following localization for the solution
$$
0.064\leq u(t)\leq 0.16,\ t\in [0,1].
$$

An application of Theorem~\ref{noext} gives that for 
$$\lambda >\frac{e+1}{e(e-1)}= 0.797,$$
there are no solutions in $K$ (the trivial solution does not satisfy the differential equation). Furthermore note that $T:P\to K$; this shows that there are no positive solutions for the BVP~\eqref{lastne} when $\lambda>\dfrac{e+1}{e(e-1)}$.
\end{exa}

\section*{Acknowledgements}
The authors would like to thank the anonymous Referee for the careful reading of the manuscript and for the constructive comments. G. Infante and P. Pietramala were partially supported by G.N.A.M.P.A. - INdAM (Italy). F. A. F. Tojo was supported by FEDER and Ministerio de Educaci\'on y Ciencia, Spain, project MTM2013-43014-P; FPU scholarship, Ministerio de Educaci\'on, Cultura y Deporte, Spain and a Fundaci\'on Barrie Scholarship. This paper was mostly written during a visit of F. Adri\'an F. Tojo to the Dipartimento di Matematica e Informatica of the Universit\`a della Calabria.  Adri\'an Tojo would like to acknowledge his gratitude towards the aforesaid Dipartimento, specially towards the co-authors of this paper, for their affectionate reception and their invaluable work and advise.

\end{document}